\documentclass[12pt,a4paper]{article}
\usepackage{amsfonts,amssymb,amsthm,amsmath,multicol,fullpage,color,dsfont,stmaryrd}
\usepackage[T1]{fontenc}
\usepackage[english]{babel}
\usepackage{setspace}
\usepackage[titletoc]{appendix}

\newcommand{\E}{\mathbb{E}}

\newcommand{\N}{\mathbb{N}}
\renewcommand{\P}{\mathbb{P}}

\newcommand{\R}{\mathbb{R}}

\newcommand{\Et}{\tilde{\E}}

\newcommand{\Pt}{\tilde{\P}}

\newcommand{\Ac}{\mathcal{A}}

\newcommand{\Fc}{\mathcal{F}}

\newcommand{\Gc}{\mathcal{G}}
\newcommand{\Hc}{\mathcal{H}}
\newcommand{\Ic}{\mathcal{I}}
\newcommand{\Lc}{\mathcal{L}}

\newcommand{\Oc}{\mathcal{O}}

\newcommand{\Tc}{\mathcal{T}}

\newcommand{\Fct}{\tilde{\Fc}}

\def \bt{\tilde{b}}
\def \sigmat{\tilde{\sigma}}
\def \pt{\tilde{p}}
\def \gammat{\tilde{\gamma}}

\def \Bt{\tilde{B}}
\def \Ht{\tilde{H}}

\def \Vt{\tilde{V}}
\def \ut{\tilde{u}}

\def \gammab{\bar{\gamma}}

\def \Om{\Omega}

\def \Omt{\tilde{\Omega}}
\def \om{\omega}

\def \omt{\tilde{\om}}

\def \Cd{\mathrm{C}}
\def \Ed{\mathrm{E}}

\newcommand{\x}{\times}
\newcommand{\ox}{\otimes}

\newcommand{\as}{\mathrm{a.s.}}
\newcommand{\eg}{\textit{e.g.}}
\newcommand{\ie}{\textit{i.e.}}

\newcommand{\X}[1]{X^{#1}}

\newcommand{\tX}[1]{\tilde{X}^{#1}}
\newcommand{\n}[1]{N^{#1}}

\newcommand{\Z}[1]{Z^{#1}}

\newcommand{\tZ}[1]{\tilde{Z}^{#1}}

\newcommand{\e}[1]{\mathrm{e}^{#1}}
\newcommand{\smallO}[1]{\ensuremath{\mathop{}\mathopen{}o\mathopen{}\left(#1\right)}}

\newcommand{\fb}{f}

\newtheorem{theorem}{Theorem}[section]
\newtheorem{lemma}[theorem]{Lemma}
\newtheorem{proposition}[theorem]{Proposition}
\newtheorem{corollary}[theorem]{Corollary}

\newtheorem{assumption}{Assumption}[section]

\newtheorem{remark}{Remark}[section]

\newcommand{\rmi}{{\rm (i)$\>\>$}}
\newcommand{\rmii}{{\rm (ii)$\>\>$}}
\newcommand{\rmiii}{{\rm (iii)$\>\>$}}
\newcommand{\rmiv}{{\rm (iv)$\>\>$}}
\newcommand{\rmv}{{\rm (v)$\>\>$}}

\title{Optimal control of branching diffusion processes: a finite horizon problem}

\author{Julien CLAISSE
	\thanks{\'Ecole Polytechnique, CMAP, Palaiseau, F-91128, France; claisse@cmap.polytechnique.fr}
}

\date{\today}

\begin{document}
 \maketitle
 
 \abstract{
 In this paper, we aim to develop the theory of optimal stochastic control for branching diffusion processes where both the movement and the reproduction of the particles depend on the control. 
 More precisely,  
 we study the problem of minimizing a criterion that is expressed as the expected value of the product of individual costs penalizing the final position of each particle.
 In this setting, we show that the value function is the unique viscosity solution of a nonlinear parabolic PDE, that is, the Hamilton-Jacobi-Bellman equation corresponding to the problem. 
 To this end, we extend the dynamic programming approach initiated by Nisio~\cite{nisio85} to deal with the lack of independence between the particles 
 as well as between the reproduction and the movement of each particle.
In particular, 
 we exploit the particular form of the optimization criterion to recover a weak form of the branching property. In addition, we provide a precise formulation and a detailed justification of the adequate dynamic programming principle. 
 }
 
 \vspace{1mm}

 {\bf Keywords.} Stochastic control, branching diffusion process, dynamic programming principle, Hamilton-Jacobi-Bellman equation, viscosity solution.

 \vspace{2mm}

 \textbf{MSC 2010.} Primary 93E20, 60J60, 60J80; secondary 49L20, 49L25, 60J70, 60J85. 
 
 \section{Introduction}
  
  Since its onset in the late 1950s, 
  the theory of optimal stochastic control and its applications have developed extensively. One of the most famous examples is perhaps the Merton portfolio problem~\cite{merton69} in where an optimal investment strategy is identified. 
  In this case, and in many applications in finance, the  stochastic process submitted to a control is a diffusion.  
  Alongside its applications, the control of diffusion processes has been a very fruitful area of research for the past decades (see, \eg,~\cite{krylov80,fleming06,pham09}). In particular, it led to the development of the theory of viscosity solutions for second-order partial differential equations (PDEs)~\cite{lions83}. Besides, an advanced theory has been elaborated to deal with optimal control of other classes of processes such as Markov chains in discrete and continuous time, piecewise deterministic Markov processes or L\'evy processes (see, \eg,~\cite{puterman94,guo09,davis93,oksendal05}). Their applications lie in a wide variety of domains including finance, operational research, computer science and epidemiology. In this paper, we aim at developing the theory of optimal control to the class of branching diffusion processes.

  The branching diffusion processes describe the evolution of a population of identical and independent particles in which each particle has a feature, \eg\ its spatial position, whose dynamic is given by a diffusion. They were first introduced by Skorohod~\cite{skorohod64} and Ikeda, Nagasawa and Watanabe~\cite{ikeda69a,ikeda69b,ikeda69c}. In particular, these authors provided new stochastic representations for semilinear second order PDEs as functional of branching diffusion processes. Since these pioneering works, the study of branching diffusion processes has developed extensively. Nowadays they are commonly seen as simple examples of measure-valued processes and were used in particular to prove existence and to approximate the so-called Dawson--Watanabe superprocesses (see, \eg, \cite{roelly86,dawson91}). They also proved valuable for numerical applications. For instance,  Henry-Labord\`ere, Tan and Touzi~\cite{henry14} developed recently an algorithm based on a Monte-Carlo method with branching diffusion processes to solve numerically semilinear PDEs and to simulate solutions of backward stochastic differential equations.

  As mentioned before, the theory of controlled diffusion processes has generated important literature. Similarly many authors have studied optimal control problems on branching processes (see, \eg, \cite{ceci99,claisse-champagnat-16}). However, regarding the control of branching diffusion processes, only two articles have been published so far regarding the control of branching diffusion processes to the best of our knowledge. On the one hand, \"Ust\"unel~\cite{ustunel81} proposed a new construction of branching diffusion processes based on martingale problems. It allows to introduce interdependence between the particles. As an application, he studied a finite horizon problem where the controls are Markovian and act solely on the drift coefficient. He proved existence of an optimal control under rather weak conditions by using a method based on the Girsanov theorem developed in~\cite{bismut76}. On the other hand, Nisio~\cite{nisio85} considered a finite horizon problem where the control acts on the drift and diffusion components of the movement and the cost function is expressed as the product of individual cost penalizing the final position of each particle.  She identified the Hamilton-Jacobi-Bellman (HJB) equation associated to this problem in the form of a nonlinear parabolic PDE and characterized the value function as its unique (viscosity) solution. 
  
  Here we intend to generalize Nisio's work, especially to allow the lifespan and the progeny distribution to depend on the position of the particle and on the control. In addition, we do not restrict to control processes that preserves the independence between the particles. Our generalization gives rise to various and profound difficulties. In particular, we have to introduce a new construction of controlled branching diffusion processes to deal with position-dependent and control-dependent reproduction. Besides,
  new arguments are needed to handle the lack of independence between the particles as well as between the reproduction and the movement of each particle. 
  
   Our method is based on the dynamic programming approach, 
  which  originates from the celebrated Bellman principle of optimality~\cite{bellman57}.
  The key step in this approach is to derive a dynamic programming principle (DPP). 
  Although it is intuitive and simple in its formulation, 
  it is very hard to give a rigourous proof in the stochastic control framework.
  Among hundreds of references, see, \eg, Krylov~\cite{krylov80}, 
  El Karoui \cite{elkaroui87}, Borkar \cite{borkar89}, Fleming and Soner~\cite{fleming06}, 
  and the more recent monographs of Pham \cite{pham09} and Touzi \cite{touzi13}  for various formulations and approaches to the DPP
  in the context of controlled diffusion processes.
  Our proof is based on the approach of Fleming and Soner~\cite{fleming06}. It relies on an existence result due to Krylov~\cite{krylov87} for smooth solutions of fully nonlinear PDE and an approximation procedure allowing to approach the value function by a sequence of smooth value functions corresponding to perturbations of the initial problem.
  In the present study, we extend the results therein to deal with controlled branching diffusion processes.
  
  The branching diffusion processes are relevant for various applications in natural science and medicine. For instance, Sawyer~\cite{sawyer76} developed applications to population genetics in order to describe the dispersion, mutation and geographical selection of the descendants of a new gene in a population of rare mutant genes.  
  More recently, Bansaye and Tran~\cite{bansaye11} created a model based on such processes to investigate the development of a paraside inside a population of dividing cells. 
  In view of the above, the control of branching diffusions processes is interesting not only from a theoritical point of view but also for its applications. For instance, it could help to improve therapeutic strategies to eliminate a virus, or at least reduce its burden, while preserving the pool of healthy cells. Other applications such as genetic selection, management of species in danger or pest control might be promising as well. 
  
  This paper is organized as follows. In the next section, we introduce the controlled branching diffusion processes and formulate our optimal control problem. We also state the main result of this work, \ie, the characterization of the value function as the unique viscosity solution of the corresponding HJB equation. The rest of the paper is devoted to the proof of this result. In Section~\ref{sec:stocal}, we collect some properties of controlled branching diffusion processes. In particular, we establish a semi-martingale decomposition of some functional of these processes. We also study the dependence of these processes w.r.t. the parameters characterizing their dynamic. Then we prove in Section~\ref{sec:hjb-smooth} that, under stringent assumptions, the value function satisfies the DPP and the HJB equation in the classical sense. In Section~\ref{sec:hjb}, we complete the proof of the main result by using an approximation procedure. We also provide a strong comparison principle for the HJB equation, which yields the uniqueness property. Finally, we establish the DPP satisfied by the value function in Section~\ref{sec:dpp}.
 
 \section{Formulation of the problem}
 
 \subsection{Controlled branching diffusions}\label{sec:branchdiff}
  
  Let $A$ be the control space that is assumed to be Polish. Denote by $\R^{d\x m}$ the set of matrices of order $d\x m$.
  
  Consider a population of identical particles such that each of them moves according to a controlled diffusion characterized by a drift $b:\R^d\x A\to\R^d$ and a diffusion coefficient $\sigma:\R^d\x A\to \R^{d\x m}$. Moreover, each particle dies at rate $\gamma:\R^d\x A\to\R_+$ and gives birth to $k\in\N$ identical particle(s) at the time and position of its death with probability $p_k:\R^d\x A\to [0,1]$. By definition, we have 
  \begin{equation*}
    \sum_{k=0}^{+\infty} {p_k\left(x,a\right)} = 1,\quad \forall\, (x,a)\in\R^d\x A.
  \end{equation*}
  
  To describe the genealogy of the population, we give a label to each particle using the Ulam-Harris-Neveu notation (see, \textit{e.g.}, \cite{bansaye11}). We introduce the set of labels:
  \begin{equation*}
    \Ic := \left\{\emptyset\right\}\cup\bigcup_{n=1}^{+\infty} \N^n.
  \end{equation*} 
  For all $i = i_1 i_2 \ldots i_n$ and $j = j_1 j_2 \ldots j_m$ in $\Ic$, we define their concatenation $ij$ by $i_1 i_2 \ldots i_n j_1 j_2 \ldots j_m$. The mother of all the particle is labeled by $\emptyset$ and when the particle $i$ gives birth to $k$ children, they bear the labels $i0, i1, \ldots, i(k-1)$.
  We also define a partial order relation on $\Ic$: we denote $j\preceq i$ (resp. $j\prec i$) if and only if there exists $j'\in\Ic$ (resp. $j'\in\Ic\setminus\{\emptyset\}$) such that $i=jj'$. 
  
  Let $(\Om,(\Fc_s)_{s\geq 0},\P)$ be a filtered probability space satisfying the usual conditions embedded with $(B^i, Q^i)_{i\in\Ic}$ a family of independent random variables such that $B^i$ is a $m$--dimensional Brownian motion and $Q^i(dt,dz)$ is a Poisson random measure on $\R_+\x\R_+$ with intensity measure $dt\,dz$, adapted to the filtration $(\Fc_s)_{s\geq 0}$.
  
  We say that $\alpha=(\alpha^i)_{i\in\Ic}$ is a control if and only if each $\alpha^i$ is a predictable process valued in $A$. Morally, each $\alpha^i$ is dedicated to control the movement and the reproduction of the particle of label $i$. Denote by $\Ac$ the collection of all controls. 

  Following the inspiration of~\cite{champagnat07, bansaye11}, we represent the population $\Z{}$ controlled by $\alpha\in\Ac$ as a measure-valued process:
  \begin{equation*}
    \Z{}_s = \sum_{i\in V_s} {\delta_{\left(i,\X{i}_s\right)}},
  \end{equation*}
  where $V_s$ contains the labels of all the particles alive and $\X{i}_s$ denotes the position of the particle $i$ at time $s$. Provided that the particle $i$ is alive, its dynamic can be roughly described as follows:
  \begin{itemize}
   \item its position $\X{i}$ is given by
    \begin{equation}\label{eq:edsi}
      d\X{i}_s = b\left(\X{i}_s,\alpha^i_s\right) ds + \sigma\left(\X{i}_s,\alpha^i_s\right) dB^i_s;
    \end{equation}
    \item its reproduction is governed by $Q^i$ and the probability that it gives birth to $k$ particle(s) in $[s,s+h]$ given $\Fc_s$ is equal to
    \begin{equation*}
      \gamma\left(\X{i}_s,\alpha^i_s\right) p_k\left(\X{i}_s,\alpha^i_s\right) h + \smallO{h}.
    \end{equation*}
  \end{itemize}

  Let $\Ed$ be the state space of the process given by
  \begin{equation*}
    \Ed:=\left\{\sum_{i\in V} {\delta_{(i,x^i)}};\  V\subset\Ic\text{ finite, } x^i\in\R^d\text{ and } i\nprec j\ \text{for all }i, j\in V\right\}.
  \end{equation*}
  We embed $\Ed$ with the weak topology. It is then a Polish space as a closed set of the space of finite measures on $\Ic\x\R^d$ (see, \eg, \cite[Sec.3.1.1]{dawson91}).
  Denote also, for all $\mu=\sum_{i\in V} {\delta_{(i,x^i)}}\in\Ed$ and $\fb=(f^i)_{i\in\Ic}$ such that $f^i:\R^d\to\R$,
  \begin{equation*}
    \langle \mu, \fb\rangle:=\sum_{i\in V} {f^i\left(x^i\right)}.
  \end{equation*}
  
  Let $\Lc^a$ be the infinitesimal generator of the (Markovian) diffusion given by the SDE~\eqref{eq:edsi} with $\alpha^i\equiv a\in A$, \ie, for all $f\in\Cd^2(\R^d)$,
  \begin{equation*}
   \Lc^a {f}(x) = \frac{1}{2}\mathrm{tr}\left(\sigma\sigma^*\left(x,a\right) D^2_x{f}\left(x\right)\right)  + b\left(x,a\right) \cdot D_x{f}\left(x\right),
  \end{equation*}
  where $D^2_x{f}$ and $D_x{f}$ denote respectively the Hessian matrix and the gradient of $f$.
  
  To characterize the dynamic of a population $\Z{}$ controlled by $\alpha=(\alpha^i)_{i\in\Ic}\in\Ac$ starting at time $t\geq0$ from initial state $\mu\in\Ed$, we consider the following SDE: for all $\fb=(f^i)_{i\in\Ic}\in\Cd^{1,2}(\R_+\x\R^d)^{\Ic}$,
 \begin{multline}\label{eq:eds}
  \langle \Z{}_s, \fb\left(s,\cdot\right)\rangle  = \langle \mu, \fb\left(t,\cdot\right)\rangle  + \int_t^s {\sum_{i\in V^{}_{\theta}} {D_x {f^i}\left(\theta,\X{i}_{\theta}\right) \sigma\left(\X{i}_{\theta},\alpha^i_{\theta}\right)   d B^i_{\theta}}} \\
  \begin{aligned}
   & + \int_t^s {\sum_{i\in V^{}_{\theta}} {\left(\partial_t{f^i}\left(\theta,\X{i}_\theta\right) + \Lc^{\alpha^i_{\theta}} {f^i}\left(\theta,\X{i}_\theta\right) \right)}} \,d {\theta} \\  
   & + \int_{(t,s]\x\R_+}  { \sum_{i\in V^{}_{\theta-}} { \sum_{k\geq 0} { \left( \sum_{l=0}^{k-1} {  f^{il}\left(\theta,\X{i}_{\theta}\right)  } - f^{i}\left(\theta,\X{i}_{\theta}\right) \right) \mathds{1}_{I_k\left(\X{i}_{\theta},\alpha^i_{\theta}\right)}\left(z\right) } \,Q^i\left(d \theta, d z\right) } },
  \end{aligned}
  \\
  \forall\, s\geq t,\ \P-\as,
 \end{multline}
  where, for all $(x,a)\in\R^d\x A$,
  \begin{gather*}
    I_{k}\left(x,a\right):=\left[\gamma\left(x,a\right) \sum_{l=0}^{k-1} {p_{l}\left(x,a\right)},\gamma\left(x,a\right) \sum_{l=0}^{k} {p_{l}\left(x,a\right)}\right),
  \end{gather*}
  with the value of an empty sum being zero by convention. Notice that $(I_{k}(x,a))_{k\in\N}$ forms a partition of the interval $[0,\gamma(x,a))$.
  
  In the SDE~\eqref{eq:eds}, the first two integrals describe the movement of the particles. In particular, one can recognize It\^o's formula applied to $f^i(\X{i}_s,\alpha^i_s)$ for each $i\in V_s$. The last integral w.r.t. the Poisson random measures characterizes the jumps of the process due to death and reproduction of the particles. Note also that the stochastic integral w.r.t. the Brownian motions is defined as follows:
  \begin{equation*}
   \sum_{i\in\Ic} \int_t^s {\mathds{1}_{i\in V^{}_{\theta}} D_x {f^i}\left(\theta,\X{i}_{\theta}\right) \sigma\left(\X{i}_{\theta},\alpha^i_{\theta}\right)   d B^i_{\theta}},
  \end{equation*}
  where the latter is well-defined under the appropriate assumptions given below.

  Let us give a first set of assumptions to ensure that the population process is well-defined as stated in the proposition below.
  \begin{assumption}\label{hyp:pop}   
	\rmi $b$ and $\sigma$ are measurable, bounded and there exists $K>0$ such that 
      \begin{equation*}
	\left|b\left(x,a\right)-b\left(y,a\right)\right| + \left|\sigma\left(x,a\right)-\sigma\left(y,a\right)\right| \leq K \left|x-y\right|, \quad \forall\, (x,y,a)\in(\R^d)^2\x A;
      \end{equation*}
    \rmii $\gamma$ is measurable and there exists $\gammab>0$ such that
    \begin{equation*}
    	\gamma\left(x,a\right) \leq \gammab, \quad\forall\, (x,a)\in\R^d\x A;
	\end{equation*}
    \rmiii the $p_k$'s are measurable and there exists $M>0$ such that
      \begin{equation*}
	\sum_{k=0}^{+\infty} {k p_{k} \left(x,a\right)} \leq M, \quad\forall\, (x,a)\in\R^d\x A.
      \end{equation*}
  \end{assumption}

  \begin{proposition}\label{prop:def}
    Let $t\in\R_+$, $\mu = \sum_{i\in V}{\delta_{(i,x^i)}}\in\Ed$ and $\alpha\in\Ac$. Under Assumption~\ref{hyp:pop}, there exists a unique (up to indistinguishability) c\`adl\`ag and adapted process $(\Z{t,\mu,\alpha}_s)_{s\geq t}$ valued in $\Ed$ satisfying the SDE~\eqref{eq:eds}. In addition, we have
    \begin{equation}\label{eq:moment}
      \E\left[\sup_{t\leq\theta\leq s}{\left\{\n{t,\mu,\alpha}_{\theta}\right\}}\right]\leq \left|V\right| \e{\gammab M (s-t)},\quad \forall s\geq t.
    \end{equation}
    where $\n{t,\mu,\alpha}_\theta$ denotes the numbers of particles alive at time $\theta$ and $\left|V\right|$ the cardinal of the set $V$.
  \end{proposition}
 
   The proof of the proposition above is postponed to Section~\ref{sec:strongbd}. It relies essentially on two arguments which follow from Assumption~\ref{hyp:pop}. First, the point~(i) ensures that there exists a unique solution to SDE~\eqref{eq:edsi}. Second, the assertions (ii) and (iii) rule out explosion, \ie, there is almost surely finitely many jumps in finite time.
  
  We conclude this section with some notations. Unless otherwise mentioned, we denote
  \begin{equation*}
    \Z{t,\mu,\alpha}_s = \sum_{i\in V^{t,\mu,\alpha}_s} {\delta_{(i,\X{i}_s)}},\quad \forall s\geq t.
  \end{equation*}
  In the important case $\mu=\delta_{(\emptyset,x)}$ with $x\in\R^d$, we simply denote $\Z{t,x,\alpha}$ and $V^{t,x,\alpha}$ instead of $\Z{t,\mu,\alpha}$ and $V^{t,\mu,\alpha}$.

 \subsection{A finite horizon problem}
  
  The aim of this section is to formulate the finite horizon problem studied in this paper. 

  Let $T>0$ be the finite horizon, $g:\R^d\to [0,1]$ and $c:\R^d\to \R_+$ be measurable maps. For all $(t,\mu,\alpha)\in\R_+\x\Ed\x\Ac$, we denote
  \begin{equation*}
  	\varGamma^{t,\mu,\alpha}_s := \exp{\left(-\int_t^s{\sum_{i\in V^{t,\mu,\alpha}_\theta} c\left(\X{i}_\theta,\alpha^i_\theta\right) \,d\theta }\right)}.
  \end{equation*}
  As before, if $\mu=\delta_{(\emptyset,x)}$ with $x\in\R^d$, we simply denote $\varGamma^{t,x,\alpha}$ instead of $\varGamma^{t,\mu,\alpha}$. 
  Define the cost function $\bar{J}:[0,T]\x\Ed\x\Ac\to[0,1]$ by
  \begin{equation*}
    \bar{J}(t,\mu,\alpha) := \E\left[\varGamma^{t,\mu,\alpha}_T \prod_{i\in V^{t,\mu,\alpha}_T} {g\left(\X{i}_T\right)}\right],
  \end{equation*}
  with the value of an empty product being one by convention. Define also both the value functions $\bar{v}:[0,T]\x\Ed\to[0,1]$ and $v:[0,T]\x\R^d\to[0,1]$ by
  \begin{equation*}
    \bar{v}(t,\mu):=\inf_{\alpha\in\Ac} {\bar{J}(t,\mu,\alpha)} \quad \text{and} \quad v(t,x) := \bar{v}(t,\delta_{(\emptyset,x)}).
  \end{equation*}
    
  The multiplicative form of the cost function is essential for the present study (see Remark~\ref{rem:cost} below). Even though it is restrictive, some relevant control problems can be expressed in such a form. For instance, if $c\equiv 0$ and $g\equiv 0$, the goal is to minimize the probability of extinction before time $T$. It is of interest in conservation biology for instance, where the controller tries to favor the survival of an endangered species (see, \textit{e.g.}, \cite{houston-namara-01}). Similarly, if $c\equiv 0$ and $g:x\mapsto\e{-|x|}$, the goal is roughly to maximize the sum of the final states of the particles. Such problems appear naturally in harvesting management, where the controller wants to increase the yield of a farming business (see, \textit{e.g.},~\cite{lenhart-workman-07}). In addition, the map $c$ allows to take into account a running cost, \textit{e.g.}, to penalize undesirable population or control states before time $T$. 
  
  \begin{remark}\label{rem:cost}
   \rmi In the uncontrolled setting, the branching property yields for all $t\in\R_+$, $\mu=\sum_{i\in V}{\delta_{(i,x^i)}}\in\Ed$ and $a\in A$,
   \begin{equation*}
   	\bar{J}\left(t,\mu,a\right) = \prod_{i\in V} {\bar{J}\left(t,\delta_{(i,x^i)},a\right)}.
   \end{equation*}
   In addition, under suitable conditions, the map $u:(t,x)\mapsto \bar{J}\left(t,\delta_{(i,x)},a\right)$ satisfies the following PDE:
   \begin{equation*}
    \partial_t u\left(t,x\right) + \Gc^a {u}\left(t,x\right) - c^a(x) u(t,x) = 0, \quad \forall\,(t,x)\in [0,T)\x\R^d, 
   \end{equation*}
   where $c^a=c(\cdot,a)$ and $\Gc^a$ is given by~\eqref{eq:cumulant} below.
   We refer the reader to~\cite{skorohod64,ikeda69c} for more details. \\
   \rmii If $g>0$, we have at our disposal another significant expression for the cost function:
    \begin{equation*}
       \bar{J}\left(t,\mu,\alpha\right) = \E\left[\exp{\left(-\int_t^T{\langle \Z{t,\mu,\alpha}_s, c^{\alpha_s}\rangle \,ds} - \langle \Z{t,\mu,\alpha}_T, -\ln{\left(g\right)}\rangle \right)}\right],
     \end{equation*}
     where $c^{\alpha_s}=(c(\cdot,\alpha^i_s))_{i\in\Ic}$.
  \end{remark}

 \subsection{Main result}
 
 The aim of this section is to state the main result of this paper, namely, the characterization of the value function as the unique viscosity solution of a nonlinear parabolic PDE. This is the so-called Hamilton-Jacobi-Bellman (HJB) equation corresponding to the optimal control problem under consideration. 
 Another important result in this paper is the corresponding dynamic programming principle (DPP), see Section~\ref{sec:dpp}.
  
 Given $a\in A$, we denote by $\Gc^a$ the operator acting on the space of functions $f\in\Cd^2(\R^d)$ bounded by $1$ as follows:
 \begin{multline}\label{eq:cumulant}
  \Gc^a {f} \left(x\right) :=  \frac{1}{2}\mathrm{tr}\left(\sigma\sigma^*\left(x,a\right) D^2_x{f}\left(x\right)\right)  + b\left(x,a\right) \cdot D_x{f}\left(x\right) \\
  +  \gamma\left(x,a\right) \left(\sum_{k\geq0} {p_{k}\left(x,a\right) {f\left(x\right)}^k} - f\left(x\right)\right).
 \end{multline}
 In the uncontrolled case, \ie\ $\alpha\equiv a$, $\Gc^a$ characterizes the cumulant semigroup and hence the law of the branching diffusion process (see, \eg, \cite{roelly90}).

 Before giving the main result, we make a new assumption regarding the regularity of the various parameters involved in the definition of the problem.
 
 \begin{assumption}\label{hyp:main}
  The maps $(p_k(\cdot,a))_{k\in\N}$, $\gamma(\cdot,a)$, $c(\cdot,a)$ and $g$ are uniformly continuous in $\R^d$, uniformly w.r.t. $a\in A$.
 \end{assumption}

 \begin{theorem}\label{th:main}
  Under Assumptions~\ref{hyp:pop} and \ref{hyp:main}, it holds for all $t\in [0,T]$ and $\mu=\sum_{i\in V} {\delta_{(i,x^i)}}\in E$,
  \begin{equation}\label{eq:branching}
  	\bar{v}(t,\mu) = \prod_{i\in V} {v(t,x^i)}.
  \end{equation}
  In addition, the value function $v$ is the unique viscosity solution valued in $[0,1]$ of 
  \begin{equation}\label{eq:hjb}
    \partial_t u\left(t,x\right) + \inf_{a\in A} {\left\{ \Gc^a {u}\left(t,x\right) - c^a(x) u(t,x) \right\}} = 0,\quad \forall\,\left(t, x\right)\in[0,T)\x\R^d,
  \end{equation}
  satisfying the terminal condition $u(T, \cdot) = g$.
 \end{theorem}
 
  The proof of this result is postponed to Section~\ref{sec:hjb}. Using a result due to Krylov~\cite{krylov87}, we show that, under stringent conditions on the parameters of the problem, the value function is the unique smooth solution to the HJB equation and satisfies the DPP. Then we use an approximation argument to approach the original value function by a sequence of smooth value functions corresponding to small perturbations of the parameters. Finally, we recover the desired results by passing to the limit, using in particular the stability property of viscosity solutions. In addition, the uniqueness property in Theorem~\ref{th:main} results from a strong comparison principle.
  
  We conclude this section by making some comments on Theorem~\ref{th:main}. First, the identity~\eqref{eq:branching} can be interpreted as a branching property satisfied by the value function. It suggests that the optimal control, if any, preserves the independence and the identicalness of the particles. In other words, the best output is reached when the controllers tries to optimize the dynamic of each particle regardless of the evolution of the others. This is a remarkable feature of the control problem under consideration, which relies on the multiplicative form of the cost function.
 
 In addition, Theorem~\ref{th:main} characterized the value function as the unique solution to the HJB equation. As a consequence, it provides a mean to compute the value function by using numerical methods for PDE such as the monotone scheme introduced by Barles and Souganidis~\cite{barles-souganidis-91}. Further, under additional assumptions, among which smoothness of $v$, it leads to the existence of an optimal (Markov) control, which consists in applying at any time $t$ the control $\hat{\alpha}(t,\X{i}_t)$ to the particle $i$ where
 \begin{equation*}
 	\Gc^{\hat{\alpha}(t,x)} {v}\left(t,x\right) - c^{\hat{\alpha}(t,x)}(x) v(t,x) = \inf_{a\in A} {\left\{ \Gc^a {v}\left(t,x\right) - c^a(x) v(t,x) \right\}}.
 \end{equation*}
 See Remark~\ref{rmk:optimal} below for more details. This confirms and strengthens our interpretation of the branching property above. 
 
 Finally, Theorem~\ref{th:main} provides an extension of a classical result from the theory of controlled diffusion processes, namely, the value function of the finite horizon problem with no running cost, terminal cost $g$ and exponential decay $c$ is the unique bounded viscosity solution of 
 \begin{equation*}
   \partial_t u\left(t,x\right) + \inf_{a\in A} {\left\{ \Lc^a {u}\left(t,x\right) - c^a(x) u(t,x) \right\}} = 0,\quad \forall\,\left(t, x\right)\in[0,T)\x\R^d,
 \end{equation*} 
 satisfying the terminal condition $u(T, \cdot) = g$ (see, \eg,~\cite{lions83,fleming06,pham09,touzi13}).

\begin{remark}
The ideas of this paper can be extended to deal with more general cost functions such as, given $h:\N\x\R^d\x A\to[0,1]$,
\begin{equation*}
 \E\left[\varGamma^{t,\mu,\alpha}_T \prod_{i\in V^{t,\mu,\alpha}_T} {g\left(\X{i}_T\right)} \prod_{j\in \bar{V}^{t,\mu,\alpha}_T} {h\left(K^j,\X{j}_{\tau^j},\alpha^j_{\tau^j}\right)}\right],
\end{equation*}
where  $\bar{V}^{t,\mu,\alpha}_T = \bigcup_{s\in[t,T]} V^{t,\mu,\alpha}_s \setminus V^{t,\mu,\alpha}_T$, $\tau^j$ and $K^j$ denote the death time and the number of children of the particle $j$ respectively. In the uncontrolled setting, the connection between this expected value and PDE's has been studied in~\cite{henry14}. 
\end{remark}

\section{Some properties of controlled branching diffusion processes}\label{sec:stocal}
 
\subsection{Existence and pathwise uniqueness}\label{sec:strongbd}

  Fix $t\in[0,T]$, $\mu=\sum_{i\in V} {\delta_{(i,x^i)}}\in\Ed$ and $\alpha=(\alpha^i)_{i\in\Ic}$ a control process. Throughout this section, we omit the indices $(t,\mu,\alpha)$ in the notations. For instance, we simply write $\Z{}$ instead of $\Z{t,\mu,\alpha}$.
  
  We are going to construct $\Z{}$ by induction on the sequence of potential jumping times. More precisely, we are going to define an increasing sequence of stopping time $(S_k)_{k\in\N}$, a sequence of random variables $(V_k)_{k\in\N}$ valued in the set of finite subsets of $\Ic$ and a sequence of processes $(\X{i};\,i\in V_k)_{k\in\N}$ such that  
 \begin{equation*}
  \Z{}_s := \sum_{k\geq 1} { \mathds{1}_{S_{k-1}\leq s < S_{k}}\sum_{i\in V_{k-1}} {\delta_{\left(i,\X{i}_s\right)}} }.
 \end{equation*}
 First, we set $(S_0,V_0):=(t,V)$ and $\X{i}_t := x^i$ for all $i\in V$. Then, for the incremental step, let $S_k$ be given by 
 \begin{equation*}
  S_{k}:=\inf{\big\{s>S_{k-1};\ \exists\, i\in V_{k-1},\ Q^i\left((S_{k-1},s]\x[0,\gammab]\right)=1\big\}}.
 \end{equation*}
 Further, for all $i\in V_{k-1}$, we define $\X{i}$ on $(S_{k-1},S_k]$ as the unique (up to indistinguishability) continuous and adapted process satisfying
 \begin{equation*}
  \X{i}_s = \X{i}_{S_{k-1}} + \int_{S_{k-1}}^s {b\left(\X{i}_{\theta},\alpha^i_{\theta}\right)} \,d\theta + \int_{S_{k-1}}^s {\sigma\left(\X{i}_{\theta},\alpha^i_{\theta}\right)} \,dB^i_{\theta},\quad \P-\as
 \end{equation*}
 Finally, we describe the branching event (if any) at time $S_k$. Let $J_k\in V_{k-1}$ be the (almost surely) unique label such that
 \begin{equation*}
  Q^{J_k}\left((S_{k-1},S_k]\x[0,\gammab]\right)=1.
 \end{equation*}
 Further, let $\zeta_{k}$ be the $[0,\gammab]$--valued random variable such that $(S_{k},\zeta_{k})$ belongs to the support of $Q^{J_k}$. Then we set
 \begin{equation*}
  V_{k} := 
  \begin{cases}
   V_{k-1}, & \text{if }\zeta_{k}\in \left[\gamma\left(\X{J_{k}}_{S_k},\alpha^{J_{k}}_{S_k}\right), \gammab\right], \\
   V_{k-1}\setminus\left\{J_{k}\right\}, & \text{if }\zeta_{k}\in I_{0}\left(\X{J_{k}}_{S_k},\alpha^{J_{k}}_{S_k}\right), \\
   \big(V_{k-1}\setminus\left\{J_{k}\right\}\big)\cup\left\{J_{k}0,\ldots,J_{k}(l-1)\right\}, & \text{if }\zeta_{k}\in I_{l}\left(\X{J_{k}}_{S_k},\alpha^{J_{k}}_{S_k}\right)\text{ for }l\geq 1.
  \end{cases}
 \end{equation*}
 In the last case, we also set $\X{i}_{S_k} := \X{J_{k}}_{S_k}$ for all $i\in V_{k}\setminus V_{k-1}$. This ends the construction of the population controlled by $\alpha$ initialized at time $t$ in state $\mu$.
 
 To ensure that the process is well-defined on $\R_+$, it remains to show that there is no explosion, \textit{i.e.},
 \begin{equation*}
  \P\left(\lim_{k\to\infty}{S_k}=+\infty\right)=1.
 \end{equation*}
 Since the jump rate per particle $\gamma$ is bounded, it is enough to show that (almost surely) the population remains finite in finite time. This is a straightforward consequence of the moment inequality~\eqref{eq:moment} of Proposition~\ref{prop:def}, which we prove below.

\begin{proof}[Proof of Proposition~\ref{prop:def}]
 Throughout this proof, we use the notations introduced above. Let us prove first that the process $\Z{}$ satisfies the SDE~\eqref{eq:eds} for every $f=(f^i)_{i\in\Ic}\in\Cd^{1,2}(\R_+\x\R^d)^\Ic$. Assume that it holds true up to time $S_{k-1}$. 
 One clearly has
 \begin{equation*}
  \langle Z_{s\wedge S_k},\fb\rangle =  \mathds{1}_{s\leq S_{k-1}} \langle Z_{s},\fb\rangle  + \mathds{1}_{S_{k-1}< s < S_{k}}\sum_{i\in V_{k-1}} {f^i\left(s, \X{i}_s\right)} + \mathds{1}_{s \geq S_{k}}\sum_{i\in V_{k}} {f^i\left(S_k, \X{i}_{S_k}\right)}.
 \end{equation*}
 We deal with the first term on the r.h.s. by induction. Let us turn now to the second one. It\^o's formula yields, for all $s\in (S_{k-1},S_{k})$,
 \begin{multline*}
  \sum_{i\in V_{k-1}} {f^i\left(s, \X{i}_s\right)} =  \sum_{i\in V_{k-1}} {\bigg(f^i\left(S_{k-1}, \X{i}_{S_{k-1}}\right)+\int_{S_{k-1}}^s { D_x {f^i}\left(\theta, \X{i}_{\theta}\right) \sigma\left(\X{i}_{\theta},\alpha^i_{\theta}\right) dB^i_{\theta}}} \\
  + \int_{S_{k-1}}^s {\left(\partial_t{f^i}\left(\theta,\X{i}_\theta\right) + \Lc^{\alpha^i_{\theta}} {f^i}\left(\theta,\X{i}_\theta\right) \right)} \,d{\theta} \bigg).
 \end{multline*}
 It remains to treat the third term. One has
 \begin{multline*}
  \sum_{i\in V_{k}} {f^i\left(S_k, \X{i}_{S_k}\right)} = \sum_{i\in V_{k-1}} {f^i\left(S_k, \X{i}_{S_k}\right)} -  \mathds{1}_{\zeta_k\in [0,\gamma(\X{J_k}_{S_k},\alpha^{J_k}_{S_k}))} f^{J_k}\left(S_k, \X{J_k}_{S_k}\right) \\
   + \sum_{l\geq 1}{\mathds{1}_{\zeta_k\in I_l\left(\X{J_k}_{S_k},\alpha^{J_k}_{S_k}\right)} \sum_{l'=0}^{l-1} {f^{J_k l'}\left(S_k, \X{J_k l'}_{S_k}\right) }}.
 \end{multline*}
 The first term on the r.h.s. is handled by It\^o's formula while the second and third terms coincide with the integral w.r.t. the Poisson random measures over $(S_{k-1},S_k]$ in~\eqref{eq:eds}. We deduce that the SDE~\eqref{eq:eds} is satisfied up to time $S_k$ and conclude this proof by induction.
 
 Let us turn now to the proof of the moment inequality~\eqref{eq:moment}. Let $(\tau_n)_{n\in\N}$ be an increasing sequence of stopping times given by
 \begin{equation*}
  \tau_n := \inf{\big\{s\geq t;\ N_s\geq n\big\}}.
 \end{equation*}
 The previous discussion ensures that the process $\Z{}$ is well-defined and satisfies~\eqref{eq:eds} up to time $\tau_n$. Applying this relation with  $f=\mathds{1}_{\Ic\x\R^d}$, we obtain
 \begin{equation*}
  N_{s\wedge\tau_n} = N_t + \int_{(t,s\wedge\tau_n]\x\R_+}  { \sum_{i\in V_{\theta-}} { \sum_{k\geq 0} { \left( k  - 1 \right) \mathds{1}_{I_k\left(\X{i}_{\theta},\alpha^i_{\theta}\right)}\left(z\right) } \,Q^i\left(d \theta, d z\right) } }.
 \end{equation*}
 It yields
 \begin{equation*}
  N^*_{s\wedge\tau_n} \leq N_t + \int_{(t,s\wedge\tau_n]\x\R_+}  { \sum_{i\in V_{\theta-}} { \sum_{k\geq 1} { \left( k  - 1 \right) \mathds{1}_{I_k\left(\X{i}_{\theta},\alpha^i_{\theta}\right)}\left(z\right) } \,Q^i\left(d \theta, d z\right) } }.
 \end{equation*}
 where $N^{*}_s := \sup_{t\leq\theta\leq s}{\{N_{\theta}\}}$ for all $s\geq t$. 
 It follows that
 \begin{align*}
  \E\left[N^*_{s\wedge\tau_n}\right] 
    & \leq N_t + \E\left[\int_t^{s\wedge\tau_n}  { \sum_{i\in V_{\theta}} { \gamma\left(\X{i}_{\theta},\alpha^i_{\theta}\right) \sum_{k\geq 1} { \left(k-1\right) p_{k}\left(\X{i}_{\theta},\alpha^i_{\theta}\right) }\ d\theta } }\right] \\
    & \leq N_t + \gammab M \E\left[\int_t^{s}  { N^*_{\theta\wedge\tau_n}\ d\theta } \right].
 \end{align*}
 Applying Gr\"onwall's lemma, we obtain
 \begin{equation*}
  \E\left[ N^{*}_{s\wedge\tau_n}\right]\leq N_t\e{\gammab M(s-t)}.
 \end{equation*}
 Since the r.h.s. does not depend on $n$, we deduce that $\tau_n$ converges almost surely to infinity. By Fatou's lemma, we conclude that the inequality~\eqref{eq:moment} holds.
 
 Finally, the pathwise uniqueness for the solution of~\eqref{eq:eds} follows from the pathwise uniqueness of SDE~\eqref{eq:edsi}. Indeed it suffices to apply~\eqref{eq:eds} with $f=\mathds{1}_{\{i\}\x\R^d}$ to characterize the dynamic of the particle $i$.
\end{proof}

\subsection{Semimartingale decomposition}

 The aim of this section is to derive a semimartingale decomposition of a specific class of functionals of the population process. 
 
 Fix $(t,\mu,\alpha)\in\R_+\x\Ed\x\Ac$. It follows from~\eqref{eq:eds} that, for all $f\in\Cd^{1,2}_b(\R_+\x\R^d)$,
 \begin{multline}\label{eq:dsmIto}
  \langle \Z{t,\mu,\alpha}_s,f\left(s,\cdot\right)\rangle = \langle \mu, f\left(t,\cdot\right)\rangle  + \int_t^s {\sum_{i\in V^{t,\mu,\alpha}_{\theta}} {D_x {f}\left(\theta, \X{i}_{\theta}\right) \sigma\left(\X{i}_{\theta},\alpha^i_{\theta}\right)   d B^i_{\theta}}} \\
  \begin{aligned}
    & + \int_t^s {\sum_{i\in V^{t,\mu,\alpha}_{\theta}} {\left(\partial_t{f}\left(\theta,\X{i}_\theta\right) + \Lc^{\alpha^i_{\theta}} {f}\left(\theta,\X{i}_\theta\right) \right) }} \,d {\theta} \\  
    & + \int_{(t,s]\x\R_+}  { \sum_{i\in V^{t,\mu,\alpha}_{\theta-}} { \sum_{k\geq 0} { \left(k-1\right) f\left(\theta, \X{i}_{\theta}\right) \mathds{1}_{I_k\left(\X{i}_{\theta},\alpha^i_{\theta}\right)} \left(z\right) } \,Q^i\left(d \theta, d z\right) } }.
  \end{aligned}
 \end{multline}

\begin{lemma}\label{lem:martingale}
 With the notations above, the process 
 \begin{equation*}
  M^{t,\mu,\alpha}_s := \int_t^s {\sum_{i\in V^{t,\mu,\alpha}_{\theta}} {D_x {f}\left(\theta, \X{i}_{\theta}\right) \sigma\left(\X{i}_{\theta},\alpha^i_{\theta}\right)   d B^i_{\theta}}},\quad s\geq t,
 \end{equation*}
 is a continuous square integrable martingale. In addition, its quadratic variation is given by
 \begin{equation*}
  \langle M^{t,\mu,\alpha}\rangle_s = \int_t^s {\sum_{i\in V^{t,\mu,\alpha}_{\theta}} {\left|D_x {f} \left(\theta, \X{i}_{\theta}\right) \sigma\left(\X{i}_{\theta},\alpha^i_{\theta}\right) \right|^2 d\theta}}.
 \end{equation*}
\end{lemma}

\begin{proof}
 For the sake of clarity, we omit the indices $(t,\mu,\alpha)$ in the notations. By definition, one has $M=\sum_{i\in\Ic} M^i$ with
 \begin{equation*}
  M^i_s := \int_t^s {\mathds{1}_{i\in V_{\theta}} {D_x {f} \left(\theta, \X{i}_{\theta}\right) \sigma\left(\X{i}_{\theta},\alpha^i_{\theta}\right)} \,dB^i_{\theta}},\quad s\geq t.
 \end{equation*}
 Since $D_x{f}$ and $\sigma$ are bounded, it is clear that $M^i$ is a continuous square integrable martingale. Besides, we recall that the space of continuous square integrable martingale is complete (see, \eg, \cite[Prop.1.5.23]{karatzas91}). Hence, if we show that $\sum_{i\in\Ic} M^i$ is absolutely convergent, it implies that $M$ is a continuous square integrable martingale. The former holds true as, for all $s\geq t$,
 \begin{align*}
  \sum_{i\in\Ic} {\E\left[\left(M^i_s\right)^2\right]} 
    & = \sum_{i\in\Ic} {\E\left[\int_t^s { {\mathds{1}_{i\in V_{\theta}}\left|D_x {f} \left(\theta, \X{i}_{\theta}\right)\sigma\left(\X{i}_{\theta},\alpha^i_{\theta}\right) \right|^2} \,d\theta}\right]}\\
    & \leq C\, \E\left[\sup_{t\leq \theta\leq s} {\left\{N_{\theta}\right\}}\right] < + \infty.
 \end{align*}
 Finally the expression for the quadratic variation follows from the independence of the Brownian motions $(B^i)_{i\in\Ic}$.
\end{proof}

For $F\in\Cd^2_b(\R)$ and $f\in\Cd^2_b(\R^d)$, we define $F_f:\Ed\to\R$ by
\begin{equation*}
 F_f\left(\mu\right) := F\left(\langle \mu, f\rangle\right).
\end{equation*}
Given $\mathsf{a}=(a^i)_{i\in\Ic}\in\Ac^{\Ic}$, we denote by $\Hc^{\mathsf{a}}$ the operator acting on the class of functions $F_f$ given by, for all $\mu=\sum_{i\in V}{\delta_{(i,x^i)}}$,
\begin{multline*}
 \Hc^{\mathsf{a}} {F_f} \left(\mu\right) := \frac{1}{2} F''\left(\langle \mu, f\rangle\right) \sum_{i\in V} {\left|D_x {f} \left(x^i\right) \sigma\left(x^i,a^i\right) \right|^2} + F'\left(\langle \mu, f\rangle\right) \sum_{i\in V} {\Lc^{a^i} {f}(x^i)} \\
 + \sum_{i\in V} { \gamma\left(x^i,a^i\right) \left(\sum_{k\geq 0} { F\Big(\langle \mu,f\rangle + \left(k-1\right) f\left(x^i\right) \Big) p_k\left(x^i,a^i\right)  - F\left(\langle \mu,f\rangle\right) } \right) }.
\end{multline*}

\begin{proposition}\label{prop:dsm}
 For $F\in\Cd^2_b(\R)$ and $f\in\Cd^{1,2}_b(\R_+\x\R^d)$, the process
  \begin{multline*}
   F_{f(s,\cdot)}\left(\Z{t,\mu,\alpha}_s\right) - F_{f(t,\cdot)}\left(\mu\right) \\
   - \int_t^s {\left(F'_{f(\theta,\cdot)}(\Z{t,\mu,\alpha}_{\theta})\langle \Z{t,\mu,\alpha}_{\theta},\partial_{t} {f(\theta,\cdot)}\rangle + \Hc^{\alpha_{\theta}} {F_{f(\theta,\cdot)}} \left(\Z{t,\mu,\alpha}_{\theta}\right)\right)} \,d\theta,\quad s\geq t,
  \end{multline*}
 is a c\`adl\`ag martingale.
\end{proposition}

\begin{proof}
 Once again we omit the indices $(t,\mu,\alpha)$ in the notation. 
 We set
 \begin{equation*}
  Q\left(\om, d s, d z, d i\right) := \sum_{i\in \Ic} {Q^i\left(\om, d s,d z\right) \delta_{i}\left(d i\right)}.
 \end{equation*}
 It is clear that $Q$ is a Poisson random measure on $\R_+\x\R_+\x\Ic$ with intensity measure $ds\, dz \sum_{i\in V} {\delta_i(di)}$. 
 Applying the generalized It\^o formula (see, \eg,~\cite[Thm.2.5.1]{ikeda89}) to~\eqref{eq:dsmIto}, we obtain
 \begin{align*}
  F_{f(s,\cdot)}\left(\Z{}_s\right) = F_{f(t,.)}\left(\mu\right) 
   & + \int_t^s { F'_{f(\theta,\cdot)}\left(\Z{}_{\theta}\right) \sum_{i\in V_{\theta}} {D_x {f} \left(\theta,\X{i}_{\theta}\right) \sigma\left(\X{i}_{\theta},\alpha^i_{\theta}\right) dB^i_{\theta}} }\\
   & + \int_t^s {\left(F'_{f(\theta,\cdot)}(\Z{}_{\theta})\langle \Z{}_{\theta},\partial_{t} {f(\theta,\cdot)}\rangle + \Hc^{\alpha_{\theta}} {F_{f(\theta,\cdot)}} \left(\Z{}_{\theta}\right)\right)} \,d\theta \\
   & + \int_{(t,s]\x\R_+\x\Ic}  {G^i_\theta(z)} \ \bar{Q}\left(d \theta, d z, d i\right),
 \end{align*}
 where $\bar{Q}(d\theta, dz, di) := Q(d\theta, dz, di) - d\theta dz \sum_{i\in V} {\delta_i(di)}$ is the compensated Poisson random measure and
  \begin{equation*}
  G_\theta^i\left(z\right) :=  \mathds{1}_{i\in V_{\theta-}} \sum_{k\geq 0} { \left( F_{f(\theta,\cdot)}\left(\Z{}_{\theta-} + \sum_{l=0}^{k-1} \delta_{\left(il,\X{i}_{\theta}\right)} \right)- F_{f(\theta,\cdot)}\left(\Z{}_{\theta-}\right)\right) \mathds{1}_{I_k\left(\X{i}_{\theta},\alpha^i_{\theta}\right)}\left(z\right)}.
 \end{equation*}
 Following the arguments of Lemma~\ref{lem:martingale}, one easily checks that the second term on the r.h.s. is a continuous square integrable martingale. Besides, since 
 \begin{equation*}
  \int_{(t,s]\x\R_+\x\Ic}  {\left|G^i_\theta(z)\right|} \ d\theta dz \sum_{i\in V} {\delta_i(di)} \leq C \sup_{t\leq \theta\leq s} {\left\{N_{\theta}\right\}},
 \end{equation*}
 it follows from Proposition~\ref{prop:def} that
 \begin{equation*}
  \E\left[\int_{(t,s]\x\R_+\x\Ic}  {\left|G^i_\theta(z)\right|} \ d\theta dz \sum_{i\in V} {\delta_i(di)}\right] < +\infty.
 \end{equation*}
 Hence, the last term is a c\`adl\`ag martingale (see, \eg, ~\cite[Sec.II.3]{ikeda89}). 
\end{proof}

 The proposition above leads to the following result, which is a key ingredient in the rest of the paper. Indeed it plays the role of It\^o's formula in the study of optimal control problems on diffusion processes.

 \begin{corollary}\label{cor:dsm}
  Let $t\in\R_+$, $\mu=\sum_{i\in V} {\delta_{(i,x^i)}}\in\Ed$ and $\alpha\in\Ac$.
  Given $u\in\Cd^{1,2}_b(\R_+\x\R^d)$ valued in $[0,1]$, the process 
  \begin{multline*}
   \varGamma^{t,\mu,\alpha}_s \prod_{i\in V^{t,\mu,\alpha}_{s}} {u\left(s, \X{i}_{s}\right)} - \prod_{i\in V} {u\left(t, x^i\right)}  \\
   - \int_t^{s} {\varGamma^{t,\mu,\alpha}_\theta \sum_{i\in V^{t,\mu,\alpha}_{\theta}} { \left(\partial_{t} {u} + \Gc^{\alpha^i_{\theta}} {u} - c^{\alpha^i_{\theta}} u\right) \left(\theta,\X{i}_{\theta}\right) \prod_{j\in V^{t,\mu,\alpha}_{\theta}\setminus\{i\}} {u\left(\theta,\X{j}_{\theta}\right)}} \,d\theta},\ s\geq t,
  \end{multline*}
  is a c\`adl\`ag martingale.
 \end{corollary}
 
 \begin{proof}
  We start by assuming that $c\equiv 0$. If there exists $\varepsilon>0$ such that $\varepsilon\leq u\leq 1$, then the result is a straightforward application of Proposition~\ref{prop:dsm} with $F(x)=\exp{(-x)}$ and $f(s,x)=-\ln{(u(s,x))}$, which belongs respectively to  $\Cd^2_b(\R_+)$ and $\Cd^{1,2}_b(\R_+\x\R^d)$. 
  Else we consider the sequence $(u_n)_{n\in\N}$  given by $u_n=\frac{u+\frac{1}{n}}{1+\frac{1}{n}}$. Clearly, $u_n\in\Cd_b^{1,2}(\R_+\x\R^d)$ and $\frac{1}{n+1}\leq u_n\leq 1$. In addition, $u_n$ and its partial derivatives converge uniformly to $u$ and its partial derivatives. Hence, the result follows by applying the first step of the proof to $u_n$ and taking the limit $n\to \infty$. In the general case $c\geq 0$, we conclude by the integration by part formula since 
  \begin{equation*}
  	\varGamma^{t,\mu,\alpha}_s 
  	= 1 - \int_t^s {\varGamma^{t,\mu,\alpha}_\theta \sum_{i\in V^{t,\mu,\alpha}_\theta} c\left(\X{i}_\theta,\alpha^i_\theta\right) \,d\theta }.
  \end{equation*}
 \end{proof}

\subsection{Dependence on parameters}

 To the best of our knowledge, even in the uncontrolled setting, there is no result in the literature concerning the regularity of branching diffusion processes w.r.t. the parameters characterizing their dynamic. The proposition below fills this gap.

 Let $\bt$, $\sigmat$, $\gammat$ and $(\pt_k)_{k\in\N}$ be respectively some drift, diffusion coefficient, death rate and progeny distribution satisfying Assumption~\ref{hyp:pop}. Given $(t,\mu,\alpha)\in\R_+\x\Ed\x\Ac$, we denote by 
 \begin{equation*}
   \tZ{t,\mu,\alpha}_s = \sum_{i\in \Vt^{t,\mu,\alpha}_s} {\delta_{(i,\tX{i}_s)}},\quad s\geq t,
 \end{equation*}
 the solution of~\eqref{eq:eds} where $b$, $\sigma$, $\gamma$ and $(p_k)_{k\in\N}$ are replaced by $\bt$, $\sigmat$, $\gammat$ and $(\pt_k)_{k\in\N}$.

 \begin{proposition}\label{prop:perturbation}
  Given $n_0\in\N$ and $\delta > 0$, there exists a modulus of continuity $\rho:\R_+\to\R_+$ such that for all $t\in[0,T]$, $\mu\in\Ed$ satisfying $\mu(\Ic\x\R^d)\leq n_0$ and $\alpha\in\Ac$,
  \begin{multline*}
    \P\left(\forall\,s\in\left[t,T\right],\ V^{t,\mu,\alpha}_s = \Vt^{t,\mu,\alpha}_s,\ \sup_{i\in V^{t,\mu,\alpha}_s} {\left\{\left|\X{i}_s-\tX{i}_s\right|\right\}}\leq \delta\right) \\
    \geq 1 - \rho\left(\|b-\tilde{b}\| + \|\sigma-\tilde{\sigma}\| + \|\gamma-\tilde{\gamma}\| + \sum_{k\in\N}{\frac{\|p_k-\tilde{p}_k\|}{2^k}} \right).
  \end{multline*}
  where $\|\cdot \|$ denotes the supremum norm.
 \end{proposition}
 
  \begin{proof}
  Once again, we omit the indices $(t,\mu,\alpha)$ in the notations. We start by observing that for any $n\geq n_0$ and $k\geq 1$,
  \begin{multline*}
  	\P\left(\forall\,s\in\left[t,T\right],\ V_s = \Vt_s,\ \sup_{i\in V_s} {\left\{\left|\X{i}_s-\tX{i}_s\right|\right\}}\leq \delta\right) \\
  	\geq \P\left(\forall\,s\in\left[t,T\right],\ V_s = \Vt_s,\ \sup_{i\in V_s} {\left\{\left|\X{i}_s-\tX{i}_s\right|\right\}}\leq \delta, S_k>T, N^{*}_T\leq n\right),
  \end{multline*}
  where $N^*_T=\sup_{t\leq s\leq T} N_s$ and $S_k$ is the $k$-th potential jumping time defined in Section~\ref{sec:strongbd}. Further we have 
   \begin{multline}\label{eq:perturbation}
     \P\left(\forall\,s\in\left[t,T\right],\ V_s = \Vt_s,\ \sup_{i\in V_s} {\left\{\left|\X{i}_s-\tX{i}_s\right|\right\}}\leq \delta, S_k>T, N^{*}_T\leq n\right) \\
     \begin{aligned}
     = 1 & - \P\left(N^{*}_T > n\right) - \P\left(S_k\leq T, N^*_T\leq n\right) \\
         & - \P\left(\forall\,s\in\left[t,T\right],\ V_s = \Vt_s,\ \sup_{t\leq s\leq T} \sup_{i\in V_s} {\left\{\left|\X{i}_s-\tX{i}_s\right|\right\}} > \delta, S_k>T, N^{*}_T\leq n\right) \\
         & - \P\left(\exists\,s\in\left[t,T\right],\ V_s \neq \Vt_s,\ S_k>T, N^{*}_T\leq n\right)    
     \end{aligned}
   \end{multline}
  By Proposition~\ref{prop:def}, the second term on the r.h.s. satisfies
  \begin{equation*}
  	\P\left(N^{*}_T > n\right) \leq \frac{C}{n},
  \end{equation*}
  where $C = n_0 e^{\gammab M T}$. 
  As for the third term, we first notice that, on the event $\{N^{*}_T \leq n\}$, $S_k$ is bounded from below by a sum of $k$ independent exponentially distributed random variables with identical parameter $n\gammab$. It follows that 
  \begin{equation*}
  	\P\left(S_k\leq T, N^{*}_T\leq n\right)\leq F_{n,k}(T)
  \end{equation*}
  where $F_{n,k}$ is the cumulative distribution of the gamma distribution with shape parameter $k$ and rate parameter $n\gammab$.
  Let us turn now to the fourth term on the r.h.s. of~\eqref{eq:perturbation}. On the event $\{S_k>T, N^*_T\leq n\}$, it is clear that for all $s\in[t,T]$,
  	\begin{equation*}
  		V_{s}\subset V\cup\Big\{ii_1\ldots i_l;\ i\in V,\ 0\leq i_1,\ldots, i_l\leq n-1,\ 1\leq l\leq k-1 \Big\}.
 	\end{equation*}
 	Since the cardinal of the set on  the r.h.s. is $C_{n,k}=n_0 \sum_{l=0}^{k-1}{n^l}$, we deduce that
  	\begin{multline*}
  		\P\left(\forall\,s\in\left[t,T\right],\ V_s = \Vt_s,\ \sup_{t\leq s\leq T} \sup_{i\in V_s} {\left\{\left|\X{i}_s-\tX{i}_s\right|\right\}} > \delta, S_k>T, N^{*}_T\leq n\right) \\
  		\leq C_{n,k} \sup_{\alpha\in\Ac} \sup_{x\in\R^d} {\P\left(\sup_{t\leq s\leq T} {\left\{\left|\X{t,x,\alpha}_s-\tX{t,x,\alpha}_{s}\right|\right\}} > \delta\right)},
  	\end{multline*}
  	where $\X{t,x,\alpha}$ and $\tX{t,x,\alpha}$ are respectively the solution of
  	\begin{gather*}
  		\X{t,x,\alpha}_s = x + \int_t^s {b\left(\X{t,x,\alpha}_\theta,\alpha^{\emptyset}_\theta\right) d\theta} + \int_t^s {\sigma\left(\X{t,x,\alpha}_\theta,\alpha^{\emptyset}_\theta\right) dB^{\emptyset}_\theta}, \quad s\geq t, \\
  		\tX{t,x,\alpha}_s = x + \int_t^s {\bt\left(\tX{t,x,\alpha}_\theta,\alpha^{\emptyset}_\theta\right) d\theta} + \int_t^s {\sigmat\left(\tX{t,x,\alpha}_\theta,\alpha^{\emptyset}_\theta\right) dB^{\emptyset}_s}, \quad s\geq t.
  	\end{gather*}
 	Under Assumption~\ref{hyp:pop}, it follows by classical arguments from the theory of (controlled) diffusions that
    \begin{equation*}
		\P\left(\sup_{t\leq s\leq T} {\left\{\left|\X{t,x,\alpha}_s-\tX{t,x,\alpha}_{s}\right|\right\}} > \delta\right) \leq \frac{C'}{\delta}\left(\|b - \bt\| + \|\sigma-\sigmat\|\right),
  	\end{equation*}
  	where the constant $C'$ does not depend on $t$, $x$ and $\alpha$. Hence, we deduce that
  	\begin{multline*}
		\P\left(\forall\,s\in\left[t,T\right],\ V_s = \Vt_s,\ \sup_{t\leq s\leq T}\sup_{i\in V_s} {\left\{\left|\X{i}_s-\tX{i}_s\right|\right\}} > \delta, S_k>T, N^{*}_T\leq n\right) \\
		\leq  \frac{C' C_{n,k}}{\delta} \left(\|b - \bt\| + \|\sigma-\sigmat\|\right).
	\end{multline*}
	It remains to deal with the fifth term on the r.h.s. of~\eqref{eq:perturbation}. First, we observe that 
	\begin{multline}\label{eq:perturbation2}
		\P\left(\exists\,s\in\left[t,T\right],\ V_s \neq \Vt_s,\ S_k>T, N^{*}_T\leq n\right) \\
			  = \sum_{l=1}^{k-1} {\P\left(\forall\, 0\leq l'\leq l-1,\ \ V_{S_{l'}} = \Vt_{S_{l'}},\ V_{S_l} \neq \Vt_{S_l},\ S_k>T, N^{*}_T\leq n\right)}
	\end{multline}
	In addition, we have for all $1\leq l\leq k-1$,
	\begin{multline*}
		\P\left(\forall\, l'\leq l-1,\ \ V_{S_{l'}} = \Vt_{S_{l'}},\ V_{S_l} \neq \Vt_{S_l},\ S_k>T, N^{*}_T\leq n\right) \\
			 \leq \P\left(\forall\, l'\leq l-1,\ \ V_{S_{l'}} = \Vt_{S_{l'}},\ \sup_{i\in V_{S_{l-1}}} {\left\{\left|\X{i}_{S_l}-\tX{i}_{S_l}\right|\right\}} > \delta,\ S_k>T, N^{*}_T\leq n\right) \\
			 + \P\left(V_{S_{l-1}} = \Vt_{S_{l-1}},\ V_{S_{l}} \neq \Vt_{S_{l}},\ \sup_{i\in V_{S_{l-1}}} {\left\{\left|\X{i}_{S_l}-\tX{i}_{S_l}\right|\right\}} \leq \delta\right).
	\end{multline*}
	Further, by the same arguments used to deal with the fourth term on the r.h.s. of~\eqref{eq:perturbation}, it holds
	\begin{multline*}
	  \P\left(\forall\, l'\leq l-1,\ \ V_{S_{l'}} = \Vt_{S_{l'}},\ \sup_{i\in V_{S_{l-1}}} {\left\{\left|\X{i}_{S_l}-\tX{i}_{S_l}\right|\right\}} > \delta,\ S_k>T, N^{*}_T\leq n\right) \\
	    \leq  \frac{C' C_{n,l}}{\delta} \left(\|b - \bt\| + \|\sigma-\sigmat\|\right).
	\end{multline*}
	In addition, we denote for all $x,y\in\R^d$ and $a\in A$,
	\begin{equation}\label{eq:intersection}
   		I^a\left(x,y\right) := \bigcup_{k\geq0} \big(I_k\left(x,a\right)\cap \tilde{I}_k\left(y,a\right)\big) \cup \big(\left[\gamma(x,a),\gammab\right]\cap[\tilde{\gamma}(y,a),\gammab]\big),
	\end{equation}
	where 
	\begin{equation*}
		\tilde{I}_k\left(y,a\right) := \left[\gammat\left(y,a\right) \sum_{l=0}^{k-1} {\pt_{l}\left(y,a\right)},\gammat\left(y,a\right) \sum_{l=0}^{k} {\pt_{l}\left(y,a\right)}\right).
	\end{equation*}
	With the notations of Section~\ref{sec:strongbd}, we have 
	\begin{multline*}
		\P\left(V_{S_{l-1}} = \Vt_{S_{l-1}},\ V_{S_l} \neq \Vt_{S_l},\ \sup_{i\in V_{S_{l-1}}} {\left\{\left|\X{i}_{S_l}-\tX{i}_{S_l}\right|\right\}} \leq \delta\right) \\
		\begin{aligned}
			& \leq \P\left(\ V_{S_{l-1}} = \Vt_{S_{l-1}}, \zeta_l \notin I^{\alpha^{J_{l}}_{S_{l}}}\left(\X{i}_{S_l},\tX{i}_{S_l}\right), \ \sup_{i\in V_{S_{l-1}}} {\left\{\left|\X{i}_{S_l}-\tX{i}_{S_l}\right|\right\}} \leq \delta\right) \\
			& \leq \E\left[\mathds{1}_{V_{S_{l-1}}=\Vt_{S_{l-1}}}  \frac{\gammab - \left|I^{\alpha^{J_l}_{S_l}}\left(\X{i}_{S_l},\X{i}_{S_l}\right)\right|}{\gammab} \prod_{i\in V_{S_{l-1}}}\mathds{1}_{\left|\X{i}_{S_l}-\tX{i}_{S_l}\right| \leq \delta}\right],
		\end{aligned}
	\end{multline*}
	where the last inequality results from the fact that $\zeta_l$ is a random variable uniformly distributed on $[0,\gammab]$, independent from the other random variables. Using Lemma~\ref{lem:inter} below, we deduce that 
	\begin{multline*}
		\P\left(V_{S_{l-1}} = \Vt_{S_{l-1}},\ V_{S_l} \neq \Vt_{S_l},\ \sup_{i\in V_{S_{l-1}}} {\left\{\left|\X{i}_{S_l}-\tX{i}_{S_l}\right|\right\}} \leq \delta\right) \\
		\leq \rho'\left(\delta + \|\gamma-\gammat\| + \sum_{k\in\N}{\frac{\|p_k-\pt_k\|}{2^k}}\right),
	\end{multline*}
	where $\rho':\R_+\to\R_+$ is a modulus of continuity.
	By~\eqref{eq:perturbation2}, we deduce that 
    	\begin{multline*}
		\P\left(\exists\,s\in\left[t,T\right],\ V_s \neq \Vt_s,\ S_k>T, N^{*}_T\leq n\right) \\
		\leq  \frac{C_{n,k-1}'}{\delta} \left(\|b - \bt\| + \|\sigma-\sigmat\|\right) + k\rho'\left(\delta + \|\gamma-\gammat\| + \sum_{k\in\N}{\frac{\|p_k-\pt_k\|}{2^k}}\right),
	\end{multline*}
	where $C_{n,k-1}'=C'\sum_{l=1}^{k-1}{C_{n,l}}$.
	We are now in a position to conclude the proof. In view of the arguments above, it follows from~\eqref{eq:perturbation} that  for any $\delta'\leq\delta$,
	\begin{multline*}
    	\P\left(\forall\,s\in\left[t,T\right],\ V_s = \Vt_s,\ \sup_{i\in V_s} {\left\{\left|\X{i}_s-\tX{i}_s\right|\right\}}\leq \delta\right) \\
    	\begin{aligned}
    		& \geq \P\left(\forall\,s\in\left[t,T\right],\ V_s = \Vt_s,\ \sup_{i\in V_s} {\left\{\left|\X{i}_s-\tX{i}_s\right|\right\}}\leq \delta'\right) \\
    		& \geq 1 - \frac{C}{n} - F_{n,k}(T) - \frac{C_{n,k}'}{\delta'} \left(\|b - \bt\| + \|\sigma-\sigmat\|\right)
    	\end{aligned} 
    	\\
    	 - k \rho'\left(\delta' + \|\gamma-\tilde{\gamma}\| + \sum_{k\in\N}{\frac{\|p_k-\tilde{p}_k\|}{2^k}} \right).
  \end{multline*}
  The conclusion follows by sending successively $\|b-\tilde{b}\| + \|\sigma-\tilde{\sigma}\| + \|\gamma-\tilde{\gamma}\| + \sum_{k\in\N}{\frac{\|p_k-\tilde{p}_k\|}{2^k}}$ to $0$, $\delta'$ to $0$, $k$ to $\infty$ and $n$ to $\infty$.
 \end{proof}

 \begin{lemma}\label{lem:inter}
  There exists a modulus of continuity $\rho:\R_+\to\R_+$ such that for all $x,y\in\R^d$ and $a\in A$,
  \begin{equation*}
   \frac{\big|I^a\left(x,y\right)\big|}{\gammab} \geq 1 - \rho\left(|x-y| + \|\gamma-\tilde{\gamma}\| + \sum_{k\in\N}{\frac{\|p_k-\tilde{p}_k\|}{2^k}}\right),
  \end{equation*}
  where $I^a(x,y)$ is given by~\eqref{eq:intersection} and $|I^a\left(x,y\right)|$ denotes its Lebesgue measure. 
 \end{lemma} 
  
 \begin{proof}
  First we observe that 
    \begin{equation*}
  	\left|\left[\gamma(x,a),\gammab\right]\cap[\tilde{\gamma}(y,a),\gammab]\right| \geq \gammab - \gamma(x,a) - \left|\gamma(x,a)-\gamma(y,a)\right| - \left\|\gamma-\gammat\right\|.
  \end{equation*}
  In addition, we have 
   \begin{multline*}
  	\left|I_k\left(x,a\right)\cap \tilde{I}_k\left(y,a\right)\right|  \geq \gamma(x,a) p_k(x,a) - 2\left(\left|\gamma(x,a)-\gamma(y,a)\right| + \left\|\gamma-\gammat\right\|\right) \\   		
  	 - 2\gammab \sum_{l=0}^{k}{\left(\left|p_k(x,a)-p_k(y,a)\right| + \left\|p_k-\pt_k\right\|\right)}.
  \end{multline*}
  Notice also that for all $K\geq 1$,
  \begin{equation*}
  	\gamma(x,a) - \gamma(x,a)\sum_{k=0}^{K-1} {p_k}(x,a) = \gamma(x,a)\sum_{k=K}^{+\infty} {p_k}(x,a) \leq \frac{\gammab M}{K}
  \end{equation*}
  where $M$ comes from Assumption~\ref{hyp:pop} (iii).
  We deduce that for all $K\geq 1$,
  \begin{multline*}
  	\begin{aligned}
  		\left|I^a(x,y)\right| 
  			& \geq  \left|\left[\gamma(x,a),\gammab\right]\cap[\tilde{\gamma}(y,a),\gammab]\right| + \sum_{k=0}^{K-1} {\left|I_k\left(x,a\right)\cap \tilde{I}_k\left(y,a\right)\right|} \\
  			& \geq \gammab  - \frac{\gammab M}{K} - (2K + 1)\left(\left|\gamma(x,a)-\gamma(y,a)\right| + \left\|\gamma-\gammat\right\|\right)
  	\end{aligned}
  	\\
  	  		- 2\gammab \sum_{k=0}^{K-1} {(K-k) \left(\left|p_k(x,a)-p_k(y,a)\right| + \left\|p_k-\pt_k\right\|\right)}.	
  \end{multline*}
  The conclusion follows immediately from Assumption~\ref{hyp:main}.
 \end{proof}

 \section{Dynamic programming: the smooth case}\label{sec:hjb-smooth}
 
 The aim of this section is to show that, under stringent conditions, the value function satisfies the DPP and the HJB equation in the classical sense. More precisely, using a result due to Krylov~\cite{krylov87}, we prove in Section~\ref{sec:hjb-smooth-sub} that there exists a classical solution to the HJB equation. Then, extending the approach of Fleming and Soner~\cite{fleming06}, we show that this solution satisfies the DPP in Section~\ref{sec:dpp-smooth}. As a consequence, we deduce that this solution coincides with the value function.
  
  \subsection{Hamilton-Jacobi-Bellman equation}\label{sec:hjb-smooth-sub}

 The proposition below ensures that, under the following assumptions, there exists a smooth solution to the HJB equation~\eqref{eq:hjb}. 
 
 \begin{assumption}\label{hyp:smooth}
 	\rmi $g\in\Cd^3_b(\R^d)$ \\
 	\rmii for $\varphi = b, \sigma, \gamma, p_k, c$, $\varphi(\cdot,a)\in \Cd^2(\R^d)$ such that $\varphi$ and its partial derivatives are bounded on $\R^d\x A$; \\
 	\rmiii there exists $K\in\N$ such that $p_k\equiv 0$ for all $k\geq K+1$ \\
 	\rmiv there exists $\underline{c}>0$ such that 
 	\begin{equation*}
	  c\left(x,a\right)\geq \underline{c},\quad \forall\, (x,a)\in\R^d\x A;
 	\end{equation*}
 	\rmv there exist $\lambda>0$ such that 
 	 \begin{equation*}
	 \sigma\sigma^*\left(x,a\right)\geq \lambda I_d,\quad \forall\, (x,a)\in\R^d\x A.
 	\end{equation*}
 \end{assumption}
 
 \begin{proposition}\label{prop:hjb-smooth}
 	Under Assumption~\ref{hyp:smooth}, there exists $u\in \Cd^{1,2}_b([0,T]\x\R^d)$ valued in $[0,1]$ such that
 	\begin{equation*}
   		\partial_t u\left(t,x\right) + \inf_{a\in A} {\left\{ \Gc^a {u}\left(t,x\right) - c^a(x) u(t,x) \right\}} = 0,\quad \forall\,\left(t, x\right)\in[0,T)\x\R^d,
 	\end{equation*} 
    satisfying the terminal condition $u(T, \cdot) = g$. In addition, $u$ and its partial derivatives are H\"older continuous on $[0,T]\x\R^d$.
 \end{proposition}
 
 \begin{proof}
 	This result is mainly an application of Theorem 6.4.4 in Krylov~\cite{krylov87}, which provides existence of smooth solutions for a class of fully nonlinear PDE. The property that the HJB equation~\eqref{eq:hjb} belongs to this class follows from  Example 6.1.8 in Krylov~\cite{krylov87}. Under Assumption~\ref{hyp:smooth}, the only issue is to find $\delta_0>0$ and $M_0>0$ such that for all $(x,a)\in\R^d\x A$,
 	\begin{gather*}
 		\gamma(x,a) \left(\sum_{k=0}^K {p_k(x,a) M_0^k} - M_0 \right) - c(x,a) M_0 \leq -\delta_0, \\
 		\gamma(x,a) \left(\sum_{k=0}^K {p_k(x,a) (-M_0)^k} + M_0 \right) + c(x,a) M_0 \geq \delta_0,
 	\end{gather*}
 	Taking $M_0=1$, both these inequalities hold with $\delta_0=\underline{c}$. Hence Theorem 6.4.4 in Krylov~\cite{krylov87} ensures that there exists $u\in \Cd^{1,2}_b([0,T]\x\R^d)$ solution to the HJB equation such that $\|u\|\leq 1$. It also ensures that the partial derivatives of $u$ are H\"older continuous. Finally, the fact that $u\geq 0$ is a straightforward consequence of the comparison principle stated in Proposition~\ref{prop:comparison} below.
 \end{proof}

  \subsection{Dynamic programming principle}\label{sec:dpp-smooth}
   
  The next proposition ensures that solution of the HJB equation given in Proposition~\ref{prop:hjb-smooth} satisfies the DPP. Denote by $\Tc_{t,T}$ the collection of all stopping times taking value in $[t,T]$.
  
  \begin{proposition}\label{prop:dpp-smooth}
  	Under Assumption~\ref{hyp:smooth}, let $u$ be as in Proposition~\ref{prop:hjb-smooth}. For all $t\in[0,T]$ and $\mu=\sum_{i\in V}{\delta_{(i,x^i)}}\in\Ed$, it holds:\\
  	\rmi for all $\alpha\in\Ac$ and $\tau\in \Tc_{t,T}$,
  	\begin{equation}\label{eq:PPD_hard}
   		\prod_{i\in V} {u\left(t,x^i\right)} \leq \E\left[\varGamma^{t,\mu,\alpha}_\tau \prod_{i\in V^{t,\mu,\alpha}_{\tau}} {u\left(\tau,X^{i}_{\tau}\right)}\right];
  	\end{equation}
  	\rmii for all $\varepsilon>0$, there exists $\alpha\in\Ac$ such that, for all $\tau\in \Tc_{t,T}$,
  	\begin{equation}\label{eq:PPD_easy}
   		\prod_{i\in V} {u\left(t,x^i\right)} + \varepsilon \geq \E\left[\varGamma^{t,\mu,\alpha}_\tau \prod_{i\in V^{t,\mu,\alpha}_{\tau}} {u\left(\tau,X^{i}_{\tau}\right)}\right].
  	\end{equation}
  \end{proposition}
  
  \begin{corollary}\label{cor:hjb-smooth}
  	Under Assumption~\ref{hyp:smooth}, the map $u$ given in Proposition~\ref{prop:hjb-smooth} coincides with the value function $v$ and the branching property~\eqref{eq:branching} is satisfied. In particular, the value function is a classical solution to the HJB equation~\eqref{eq:hjb} and satisfies the DPP~\eqref{eq:PPD_hard}--\eqref{eq:PPD_easy}. 
  \end{corollary}

  \begin{proof}
    By applying Proposition~\ref{prop:dpp-smooth} with $\tau\equiv T$, we deduce that
    \begin{equation*}
    	\bar{v}(t,\mu)=\prod_{i\in V} {u\left(t,x^i\right)}.
    \end{equation*}
    Taking $\mu=\delta_{(\emptyset,x)}$, we conclude that $u=v$ and so the identity above turns out to be the branching property~\eqref{eq:branching}.
  \end{proof}
  
  \begin{proof}[Proof of Proposition~\ref{prop:dpp-smooth}]
  	Once again, we omit the indices $(t,\mu,\alpha)$ in the notations. Let us start by proving (i). Applying Corollary~\ref{cor:dsm}, we obtain
  	\begin{multline*}
	   \E\left[\varGamma_\tau \prod_{i\in V_{\tau}} {u\left(\tau, \X{i}_{\tau}\right)}\right] = \prod_{i\in V} {u\left(t,x^i\right)} \\
   		+ \E\left[\int_t^{\tau} {\left(\varGamma_s \sum_{i\in V_s} { \left(\partial_{t} {u} + \Gc^{\alpha^i_s} {u} - c^{\alpha_s^i} u \right) \left(s,\X{i}_s\right) \prod_{j\in V_s \setminus \{i\}} {u\left(s,\X{j}_s\right)}}\right)} \,ds\right].
  	\end{multline*}
  	Since $u$ is a nonnegative solution to the HJB equation, we deduce that (i) is satisfied. Let us turn now to the proof of (ii). The idea is to construct a near optimal control. Fix $\varepsilon>0$.  In view of Proposition~\ref{prop:hjb-smooth}, the map $u$ and its partial derivatives are uniformly continuous in $[0,T]\x\R^d$. Hence, there exists $\delta>0$ such that for all $|s-s'|\leq\delta$, $|y-y'|\leq\delta$ and $a\in A$,
  	\begin{equation*}
  		\left|(\partial_t u + \Gc^a {u} - c^a u)(s,y) - (\partial_t u + \Gc^a {u} - c^a u)(s',y')\right| \leq \frac{\varepsilon}{2}.
	\end{equation*}
	Let $(B_m)_{m\in\N}$ be a partition of $\R^d$ in Borel sets of diameter less than $\frac{\delta}{2}$. We choose an element $y_m$ in each $B_m$. Similarly, let $t=s_0<s_1<\cdots<s_N=T$ be a subdivision of $[0,T]$ such that $s_{n+1}-s_n=\frac{T-t}{N}\leq \delta$. Then, for each $(n,m)$, we take $a_{n,m}\in A$ such that 
  	\begin{equation*}
  		(\partial_t u + \Gc^{a_{n,m}} {u} - c^{a_{n,m}} u)(s_n,y_m) \leq \frac{\varepsilon}{2}.
	\end{equation*}	
  	Hence, it holds for all $s\in[s_n,s_{n+1}]$ and $|y-y_m|\leq \delta$,
  	\begin{equation}\label{eq:proof-dpp-smooth}
  		(\partial_t u + \Gc^{a_{n,m}} {u} - c^{a_{n,m}} u)(s,y) \leq \varepsilon.
	\end{equation}
	We define a near optimal control process as follows:
	\begin{equation*}
		\alpha_s^i(\om) = a_{n,m}, \qquad \text{if }s\in(s_n,s_{n+1}],\ \X{i}_{s_n}\in B_m,
	\end{equation*}
 	where we extend the trajectory of the particle $i$ before its birth by the trajectory of its ancestors, \ie, we set $\X{i}_{s} := \X{j}_s$ whenever $i\succeq j\in V_{s}$.
	Applying Corollary~\ref{cor:dsm}, we get
	\begin{multline*}\label{eq:dpp-proof}
	   \E\left[\varGamma_\tau \prod_{i\in V_{\tau}} {u\left(\tau, \X{i}_{\tau}\right)}\right] = \prod_{i\in V} {u\left(t,x^i\right)}  \\
   		+ \E\left[\int_t^{\tau} {\left(\varGamma_s \sum_{i\in V_s} { \left(\partial_{t} {u} + \Gc^{\alpha^i_s} {u} - c^{\alpha_s^i} u \right) \left(s,\X{i}_s\right) \prod_{j\in V_s \setminus \{i\}} {u\left(s,\X{j}_s\right)}}\right)} \,ds\right].
  	\end{multline*}
  	It remains to show that the second term on the r.h.s. is bounded from above by a quantity that can be made arbitrary small. Denote
  	\begin{equation*}
  		F_\delta:=\left\{\left|\X{i}_s-\X{i}_{s_n}\right|\leq\frac{\delta}{2}, i\in V_s, s\in(s_n,s_{n+1}], n=0,\ldots,N-1\right\}.
  	\end{equation*}
  	By~\eqref{eq:proof-dpp-smooth}, it holds
  	\begin{multline*}
  		\E\left[\int_t^{\tau} {\varGamma_s \left(\sum_{i\in V_s} { \left(\partial_{t} {u} + \Gc^{\alpha^i_s} {u} - c^{\alpha_s^i} u \right) \left(s,\X{i}_s\right) \prod_{j\in V_s \setminus \{i\}} {u\left(s,\X{j}_s\right)}}\right)} \,ds\right] \\
  		\leq C \left(\varepsilon  + \sup_{a\in A} \left\{\|\partial_{t} {u} + \Gc^{a} {u} - c^{a} u \|\right\} \P\left(\Om\setminus F_\delta\right)\right),
  	\end{multline*}
  	where $C=(T-t)e^{\gammab M (T-t)}$. It remains to evaluate $\P\left(\Om\setminus F_\delta\right)$:
  	\begin{align*}
  		\P\left(\Om\setminus F_\delta\right) 
  			& = \P\left(\sup_{0\leq n\leq N-1} \sup_{s_n < s\leq s_{n+1}} \sup_{i\in V_s} {\left\{\left|\X{i}_s-\X{i}_{s_n}\right|\right\}} > \frac{\delta}{2}\right)\\
  			& \leq \P\left(\sup_{0\leq n\leq N-1} \sup_{s_n < s\leq s_{n+1}} \sup_{i\in V_s} {\left\{\left|\X{i}_s-\X{i}_{s_n}\right|\right\}} > \frac{\delta}{2}, S_{k} > T\right) + \P\left(S_k\leq T\right),
  	\end{align*}
  	where $S_k$ is the $k$-th potential jumping time defined in Section~\ref{sec:strongbd}. In view of Assumption~\ref{hyp:smooth} (iii), 
  	it holds for all $s\in [t,T]$,
  	\begin{equation*}
  		V_s\subset V\cup\Big\{i i_1\ldots i_l;\ i\in V,\ 0\leq i_1,\ldots,i_l\leq K-1,\ 1\leq l\leq k - 1 \Big\}.
 	\end{equation*}
 	Since the cardinal of the set on the r.h.s. is $C_k=|V|\sum_{l=0}^{k-1} { K^l}$, we deduce that 
  	\begin{multline*}
  		\P\left(\sup_{0\leq n\leq N-1} \sup_{s_n \leq s\leq s_{n+1}} \sup_{i\in V_s} {\left\{\left|\X{i}_s-\X{i}_{s_n}\right|\right\}} > \frac{\delta}{2}, S_{k} > T\right) \\
  		\leq N C_k \sup_{0\leq n\leq N-1} \sup_{i\in V} \sup_{\alpha\in\Ac} {\P\left(\sup_{s_n < s\leq s_{n+1}} {\left\{\left|\X{t,x^i,\alpha}_s-\X{t,x^i,\alpha}_{s_n}\right|\right\}} > \frac{\delta}{2}\right)},
  	\end{multline*}
  	where $\X{t,x,\alpha}$ is the solution of
  	\begin{equation*}
  		\X{t,x,\alpha}_s = x + \int_t^s {b\left(\X{t,x,\alpha}_\theta,\alpha^{\emptyset}_\theta\right) d\theta} + \int_t^s {\sigma\left(\X{t,x,\alpha}_\theta,\alpha^{\emptyset}_\theta\right) dB^{\emptyset}_\theta}, \quad s\geq t.
  	\end{equation*}
  	Under Assumption~\ref{hyp:pop}, it follows by classical arguments from the theory of (controlled) diffusions that
    \begin{equation*}
		\P\left(\sup_{s_n < s\leq s_{n+1}} {\left\{\left|\X{t,x,\alpha}_s-\X{t,x,\alpha}_{s_n}\right|\right\}} > \frac{\delta}{2}\right) \leq \frac{C' (s_{n+1} -s_n)^2}{\delta^4} = \frac{C' (T -t)^2}{N^2 \delta^4},
  	\end{equation*}
  	where the constant $C'$ does not depend on $x$ and $\alpha$. Hence, we deduce that 
  	\begin{equation*}
  		\P\left(\Om\setminus F_\delta\right) \leq \frac{C'_k}{N \delta^4} + \P\left(S_k\leq T\right),
  	\end{equation*}
  	where $C'_k =  C_k C' (T-t)^2$. In addition, $S_k-t$ is bounded from below by the sum of $k$ independent exponentially distributed random variables with parameters $(\gammab(|V| + l K))_{0\leq l\leq k-1}$. It follows that $\P\left(S_k\leq T\right)$ converges to $0$, uniformly w.r.t. $\alpha$, as $k$ goes to $+\infty$. Hence, $\P\left(\Om\setminus F_\delta\right)$ vanishes, uniformly w.r.t. $\alpha$, as $N$ tends to $\infty$. This ends the proof.
  \end{proof}
  
  \begin{remark}\label{rmk:optimal}
   If we assume further that the parameters are continuous in $a$ and that there exists a solution to 
   \begin{equation*}
    d X_s = b\left(X_s,\hat{\alpha}(s,X_s)\right) ds + \sigma\left(X_s,\hat{\alpha}(s,X_s)\right) d B_s,
   \end{equation*}
   where $\hat{\alpha}:[0,T]\x\R^d\to A$ is given by
   \begin{equation*}
      \Gc^{\hat{\alpha}(s,x)} {v}\left(s,x\right) - c^{\hat{\alpha}(s,x)}(x) v(s,x) = \inf_{a\in A} {\left\{ \Gc^a {v}\left(s,x\right) - c^a(x) v(s,x) \right\}}.
   \end{equation*}
   Then we can show by a classical verification argument that an optimal control consists in applying at any time $s$ the control $\hat{\alpha}(s,\X{i}_s)$ to each particle $i$. We refer the reader to Fleming and Soner~\cite{fleming06} for conditions to ensure the existence of a (weak) solution to the SDE above.
  \end{remark}

 \section{Proof of Theorems~\ref{th:main}}\label{sec:hjb}
 
 The aim of this section is to show that the value function is the unique viscosity solution of the corresponding HJB equation. First, we derive the uniqueness property from a strong comparison principle stated in Section~\ref{sec:hjbuniq}. Then, we prove in Section~\ref{sec:hjbexist} that the value function satisfies~\eqref{eq:hjb} in the viscosity sense by approximation with smooth value functions corresponding to small perturbations of the initial problem. In Section~\ref{sec:dpp}, we derive the DPP by using the same approximation procedure.

 \subsection{Comparison principle}\label{sec:hjbuniq}
 
  In this section, we give a strong comparison principle for the HJB equation. 
 To the best of our knowledge, the comparison principle for a parabolic PDE such as~\eqref{eq:hjb} appears solely in~\cite{zhan99}. However, for the sake of completeness, we provide another proof, which is based on an extension of the arguments in~\cite{pham09}.

 \begin{proposition}\label{prop:comparison}
  Let $u_1$ and $u_2$ be respectively viscosity subsolution and supersolution valued in $[-1,1]$ of the HJB equation~\eqref{eq:hjb}. If $u_1(T,\cdot)\leq u_2(T,\cdot)$ in $\R^d$, then $u_1\leq u_2$ in $[0,T]\x\R^d$. In particular, there exists at most one viscosity solution valued in $[-1,1]$ to the HJB equation~\eqref{eq:hjb}.
 \end{proposition}
 
 \begin{proof} 
  \textsl{First step.} Let $H:\R^d\x\R\x\R^d\x\R^{d\x d}\to \R$ be the Hamiltonian of the HJB equation, \ie,
  \begin{equation}\label{eq:hamiltonian}
      H(x,r,p,M) := \inf_{a\in A} \left\{b(x,a)\cdot p + \frac{1}{2}\mathrm{tr}\left(\sigma\sigma^*(x,a) M\right) + G^a\left(x,r\right) - c(x,a) r\right\},
  \end{equation}
  where $G^a:\R^d \x\R\to \R$ is given by
  \begin{equation*}
   G^a\left(x,r\right) := \gamma(x,a) \left(\sum_{k\geq0} {p_k(x,a)r^k} - r\right), \quad \forall r\in[-1,1],
  \end{equation*}
  and $G^a(x,r)=G^a(x,1)$ for all $r\geq 1$, $G^a(x,r)=G^a(x,-1)$ for all $r\leq -1$.
  Let us show first that there exists $K>0$ such that for all $x\in\R^d$, $p\in\R^d$, $M\in\R^{d\x d}$ and $r_1\leq r_2$,
  \begin{equation*}
   H(x,r_2,p,M) - H(x,r_1,p,M) \leq K (r_2 - r_1).
  \end{equation*}
  We start by observing that for all $x\in\R^d$, $a\in A$ and $r_1,r_2\in[-1,1]$, 
  \begin{equation}\label{eq:G}
  	\left|\sum_{k\geq 0} {p_k(x,a) \left(r_2^k -  r_1^k\right)}\right| = \left|r_2-r_1\right|\left|\sum_{k\geq 0} {p_k(x,a) \sum_{l=0}^{k-1} r_1^l r_2^{k-l-1}}\right| \leq M \left|r_2-r_1\right|,
  \end{equation}
  where the constant $M$ comes from the point (iii) of Assumption~\ref{hyp:pop}. Since $c$ and $\gamma$ are nonnegative, it follows that for all $r_1\leq r_2$, 
  \begin{equation*}
   H(x,r_2,p,M) - H(x,r_1,p,M)
    \leq \sup_{a\in A} {\left\{G^a\left(x,r_2\right) - G^a\left(x,r_1\right)\right\}}
    \leq K (r_2 - r_1),
  \end{equation*}
  where $K=\gammab M$. 
  Now let $\ut_1:=\e{\lambda t} u_1$ and $\ut_2:=\e{\lambda t} u_2$ with $\lambda=K+1$. One easily checks that $\ut_1$ and $\ut_2$ are respectively viscosity subsolution and supersolution of
  \begin{equation}\label{eq:hjbcomp}
   \partial_t u(t,x) + \Ht\left(t,x,u(t,x),D_x{u}(t,x), D^2_x{u}(t,x)\right) = 0,\quad \forall (t,x)\in[0,T)\x\R^d,
  \end{equation}
  where $\Ht:[0,T]\x\R^d\x\R\x\R^d\x\R^{d\x d}\to\R$ is given by
  \begin{equation*}
   \Ht(t,x,r,p,M) := -\lambda r + \e{\lambda t} H\left(x,\e{-\lambda t} r, \e{-\lambda t} p, \e{-\lambda t} M\right).
  \end{equation*}
  From the calculation above, it follows that for all $r\leq s$,
  \begin{equation}\label{eq:comp}
   \Ht(t,x,s,p,M) - \Ht(t,x,r,p,M) \leq - \left(s - r\right).
  \end{equation}
  In the rest of the proof, we are going to prove the comparison principle for~\eqref{eq:hjbcomp}, \ie, $\ut_1\leq \ut_2$, which clearly yields $u_1\leq u_2$.
  
  \noindent \textsl{Second step.} Now we prove that we can always suppose that $\ut_1-\ut_2$ reaches its maximum in a compact subset of $[0,T]\x\R^d$. Indeed, if it is not the case,  we can replace $\ut_2$ by $\ut_2^{\varepsilon}:=\ut_2 + \varepsilon\phi$ with $\varepsilon>0$ and $\phi(t,x):=\e{-\rho t}\left(1+|x|^2\right)$. Let us show that $\ut_2^{\varepsilon}$ is a viscosity supersolution of~\eqref{eq:hjbcomp} if $\rho$ is sufficiently large. 
  
  Fix $(t,x)\in [0,T)\x\R^d$ and let $\psi\in\Cd^{1,2}([0,T]\x\R^d)$ be such that $(t,x)$ is a minimum point of $\ut^\varepsilon_2-\psi$ and $\ut^\varepsilon_2(t,x)=\psi(t,x)$. We want to prove that, for $\rho$ sufficiently large, 
  \begin{equation*}
   \partial_t \psi (t,x) + \Ht\left(t,x,\psi(t,x),D_x{\psi}(t,x), D^2_x{\psi}(t,x)\right) \leq 0.
  \end{equation*}
  First, since $\ut_2$ is a viscosity supersolution of~\eqref{eq:hjbcomp}, one has
  \begin{equation*}
   \partial_t \psi^{\varepsilon}(t,x) + \Ht\left(t,x,\psi^{\varepsilon}(t,x),D_x{\psi^{\varepsilon}}(t,x), D^2_x{\psi^{\varepsilon}}(t,x)\right) \leq 0,
  \end{equation*}
  with $\psi^{\varepsilon}:=\psi-\varepsilon\phi$. Further, it follows from~\eqref{eq:comp} that 
  \begin{multline*}
   \Ht\left(t,x,\psi(t,x),D_x{\psi}(t,x), D^2_x{\psi}(t,x)\right) - \Ht\left(t,x,\psi^{\varepsilon}(t,x),D_x{\psi^{\varepsilon}}(t,x), D^2_x{\psi^{\varepsilon}}(t,x)\right)\\
   \begin{aligned}
    & \leq \varepsilon\left(- \phi + \sup_{a\in A} {\left\{b(x,a)\cdot D_x{\phi}(t,x) + \frac{1}{2}\mathrm{tr}\left(\sigma\sigma^*(x,a) D^2_x{\phi}(t,x)\right) \right\}}\right) \\
    & \leq \left(C - 1\right) \varepsilon \phi,
   \end{aligned}
  \end{multline*}
  for some constant $C>0$. We deduce that
  \begin{equation*}
   \partial_t \psi(t,x) + \Ht\left(t,x,\psi(t,x),D_x{\psi}(t,x), D^2_x{\psi}(t,x)\right) \leq \left(C - 1 - \rho\right) \varepsilon \phi.
  \end{equation*} 
  Hence, if $\rho\geq C-1$, $\ut^\varepsilon_2$ is a viscosity supersolution of~\eqref{eq:hjbcomp}.
  
  \noindent\textsl{Third step.} To conclude, we argue by contradiction to show that $\ut_1\leq \ut^{\varepsilon}_2$, which gives the desired result by sending $\varepsilon$ to zero. Assume that 
  \begin{equation*}
   M := \sup_{[0,T]\x\R^d} {\left\{\ut_1-\ut^\varepsilon_2\right\}} > 0.
  \end{equation*}
  Since $\ut_1(T,\cdot)\leq\ut^{\varepsilon}_2(T,\cdot)$ and $\lim_{|x|\to+\infty} {\sup_{[0,T]} {\left\{\ut_1(\cdot,x)-\ut^\varepsilon_2(\cdot,x)\right\}}}=-\infty$, there exists an open bounded set $\Oc$ of $\R^d$ such that the supremum above is attained in $[0,T)\x \Oc$ and $\sup_{[0,T]\x\partial\Oc} {\{\ut_1 - \ut^\varepsilon_2\}}<M$. Now we use the classical dedoubling variable technique. Consider, for any $\delta>0$, the function 
  \begin{equation*}
   \phi_{\delta}(t,s,x,y) := \frac{1}{\delta} \left(\left|t-s\right|^2 + \left|x-y\right|^2\right),
  \end{equation*}
  and denote
  \begin{equation*}
   M_{\delta} := \max_{[0,T]^2\x\bar{\Oc}^2} {\left\{\ut_1(t,x) - \ut^{\varepsilon}_2(s,y) - \phi_{\delta}(t,s,x,y)\right\}}.
  \end{equation*}
  Let $(t_\delta,s_\delta,x_\delta,y_\delta)$ be an argument of the maximum above. It is well known (see, \eg,~\cite[Lem.3.1]{crandall92}) that
  \begin{equation*}
   \lim_{\delta\to 0} {M_{\delta}} = M ~~ \text{and} ~~ \lim_{\delta\to 0} {\phi_{\delta}\left(t_\delta,s_\delta,x_\delta,y_\delta\right)} = 0. 
  \end{equation*}
  In particular, it follows that $(t_\delta,s_\delta,x_\delta,y_\delta)\in[0,T)^2\x \Oc^2$  for $\delta$ small enough. In view of the celebrated Ishii lemma (see, \eg,~\cite[Thm.8.3]{crandall92}), there exist $X,Y\in\R^{d\x d}$ such that 
  \begin{equation}\label{eq:ishii}
   \begin{pmatrix}
    X & 0 \\
    0 & -Y
   \end{pmatrix}
   \leq 
   \frac{3}{\delta}
   \begin{pmatrix}
    I_d & -I_d \\
    -I_d & I_d
   \end{pmatrix}
   ,
  \end{equation}
  and 
  \begin{gather*}
   \frac{1}{\delta}\left(t_{\delta}-s_{\delta}\right) +  \Ht\left(t_{\delta},x_{\delta},\ut_1(t_{\delta},x_{\delta}),\frac{1}{\delta}\left(x_{\delta}-y_{\delta}\right), X\right) \geq 0, \\
   \frac{1}{\delta}\left(t_{\delta}-s_{\delta}\right) +  \Ht\left(s_{\delta},y_{\delta},\ut^{\varepsilon}_2(s_{\delta},y_{\delta}),\frac{1}{\delta}\left(x_{\delta}-y_{\delta}\right), Y\right) \leq 0.
  \end{gather*}
  From~\eqref{eq:comp} and the two inequalities above, it follows that
  \begin{align*}
   M &\leq M_ {\delta} \leq \ut_1(t_{\delta},x_{\delta}) - \ut^{\varepsilon}_2\left(s_{\delta},y_{\delta}\right) \\
     & \leq \Ht\left(t_{\delta},x_{\delta},\ut^{\varepsilon}_2\left(s_{\delta},y_{\delta}\right),\frac{1}{\delta}\left(x_{\delta}-y_{\delta}\right), X\right) - \Ht\left(t_{\delta},x_{\delta},\ut_1(t_{\delta},x_{\delta}) ,\frac{1}{\delta}\left(x_{\delta}-y_{\delta}\right), X\right) \\
     & \leq \Ht\left(t_{\delta},x_{\delta},\ut^{\varepsilon}_2\left(s_{\delta},y_{\delta}\right),\frac{1}{\delta}\left(x_{\delta}-y_{\delta}\right), X\right) - \Ht\left(s_{\delta},y_{\delta},\ut^{\varepsilon}_2(s_{\delta},y_{\delta}),\frac{1}{\delta}\left(x_{\delta}-y_{\delta}\right), Y\right).
  \end{align*}
  In view of Lemma~\ref{lem:moduluscont} below, it yields
  \begin{equation*}
   0 < M \leq \rho\left(\left|t_{\delta}-s_{\delta}\right| + \left|x_{\delta}-y_{\delta}\right| + \frac{1}{\delta}\left|x_{\delta}-y_{\delta}\right|^2\right)\Big(1 + \ut^{\varepsilon}_2\left(s_{\delta},y_{\delta}\right)\Big),
  \end{equation*}
  where $\rho:\R_+\to\R_+$ is a modulus of continuity.
  Since $\ut^{\varepsilon}_2$ is bounded on $[0,T]\x\Oc$,  the term on the r.h.s tends to zero as $\delta$ goes to infinity, which leads to a contradiction. It follows that $\ut_1\leq \ut_2^{\varepsilon}$. By sending $\varepsilon$ to zero, we deduce that $\ut_1\leq \ut_2$ and so $u_1\leq u_2$.
 \end{proof}
 
  \begin{lemma}\label{lem:moduluscont}
  With the notation of the proof above, there exists a modulus of continuity $\rho:\R_+\to\R_+$ such that
  \begin{multline*}
   \Ht\left(t,x,r,\frac{1}{\delta}(x-y),X\right) - \Ht\left(s,y,r,\frac{1}{\delta}(x-y),Y\right)\\
    \leq \rho\left(\left|t-s\right| + \left|x-y\right| + \frac{1}{\delta}\left|x-y\right|^2\right)\left(1+r\right),
  \end{multline*}
  for all $t,s\in[0,T]$, $x,y\in\R^d$, $r\in\R$, $X,Y\in\R^{d\x d}$ satisfying~\eqref{eq:ishii}. 
 \end{lemma}
 
 \begin{proof}
  Recall the well-known calculation, for every $U,V\in\R^{d\x m}$,
  \begin{align*}
   \mathrm{tr}\left(UU^*X-VV^*Y\right) 
    & = \mathrm{tr}\left( 
      \begin{pmatrix}
	UU^* & UV^* \\
	VU^* & VV^*
      \end{pmatrix}
      \begin{pmatrix}
	X & 0 \\
	0 & -Y
      \end{pmatrix}    
      \right) \\
    & \leq  \frac{3}{\delta} \mathrm{tr}\left( 
      \begin{pmatrix}
	UU^* & UV^* \\
	VU^* & VV^*
      \end{pmatrix}
      \begin{pmatrix}
	I_d & I_d \\
	I_d & -I_d
      \end{pmatrix}    
      \right)\\
    &\leq \frac{3}{\delta} \mathrm{tr}\left(\left(U-V\right)\left(U-V\right)^*\right).
  \end{align*}
  It follows that there exists $C>0$ such that for all $t,s\in[0,T]$, $x,y\in\R^d$, $X,Y\in\R^{d\x d}$ satisfying \eqref{eq:ishii},
  \begin{multline*}
   \Ht\left(t,x,r,\frac{1}{\delta}(x-y),X\right) - \Ht\left(s,y,r,\frac{1}{\delta}(x-y),Y\right) \leq \frac{C}{\delta} \left|x-y\right|^2 \\
       + \sup_{a\in A}{\left\{\left(c\left(y,a\right) - c\left(x,a\right)\right) r + \e{\lambda t} G^a\left(x,r\e{-\lambda t}\right) - \e{\lambda s} G^a\left(y,r\e{-\lambda s}\right)\right\}}.
  \end{multline*}
  Further, by using~\eqref{eq:G} and $\|G\|\leq 2\gammab$,  a straightforward calculation yields 
  \begin{multline*}
   \left|\e{\lambda t} G^a\left(x,r\e{-\lambda t}\right) - \e{\lambda s} G^a\left(y,r\e{-\lambda s}\right)\right| \leq 2 \gammab\left|\e{\lambda t} -\e{\lambda s}\right| + 2 \e{\lambda T} \left|\gamma(x,a)-\gamma(y,a)\right|  \\
   +  \gammab \e{\lambda T} \sum_{k\geq 0} {\left|p_k(x,a)-p_k(y,a)\right|} +  \gammab\left( M + 1\right) \e{2\lambda T}\left|\e{-\lambda t} - \e{-\lambda s}\right|.
  \end{multline*}
  The conclusion follows immediately by Assumption~\ref{hyp:main}.
 \end{proof}

 \subsection{Approximation procedure}\label{sec:hjbexist}
	
 In this section, we show that the value function satisfies the branching property~\eqref{eq:branching} and the HJB equation~\eqref{eq:hjb} in the viscosity sense. The idea of the proof is to approximate the value function $v$ by a sequence of smooth value function $(v_n)_{n\in\N}$ corresponding to small perturbations of the original problem. 
 
 Let $(\rho_n)_{n\in\N}$ be a family of mollifiers, \eg, $\rho_n(x)=n^d\rho(nx)$ where
 \begin{equation}
 	\rho(x) = \exp{\left(-\frac{1}{1-|x|^2}\right)} \mathds{1}_{|x| < 1}.
 \end{equation}
 We construct smooth approximations of the parameters as follows: $b_n(\cdot,a)= b(\cdot,a)*\rho_n$, $\sigma_n(\cdot,a)= \sigma(\cdot,a)*\rho_n$, $\gamma_n(\cdot,a)= \gamma(\cdot,a)*\rho_n$, $g_n= g*\rho_n$, $c_n(\cdot,a)= c(\cdot,a)*\rho_n + \frac{1}{n}$, $p_{n,k}(\cdot,a)= p_k(\cdot,a)*\rho_n$ for all $k<n$ and 
 \begin{equation*}
  p_{n,n} = 1-\sum_{k=0}^{n-1} {p_{n,k}}.
 \end{equation*}
 Clearly, these parameters satisfy Assumption~\ref{hyp:smooth} (i)--(iv). The uniform ellipticity condition (v) is more delicate to obtain. To this end, we start by enlarging the probability space.
 
   Let $(\Omt,(\Fct_s)_{s\geq 0},\Pt)$ be a filtered probability space 
embedded with $(\Bt^i)_{i\in\Ic}$ a family of independent $d$-dimensional Brownian motions. Define the enlarged probability space $(\Om\x\Omt,(\Fc_s\ox\Fct_s)_{s\geq0},\P\ox\Pt)$ and, by abuse of notations, for all $(\om,\omt)\in\Om\x\Omt$, $B^i(\om,\omt)=B^i(\om)$, $Q^i(\om,\omt)=Q^i(\om)$ and $\Bt^i(\om,\omt)=\Bt^i(\omt)$. Clearly, $(B^i,\Bt^i,Q^i)_{i\in\Ic}$ is a family of independent Brownian motions and Poisson random measures in the enlarged probability space. Denote by $\tilde{\Ac}$ be the collection of $\alpha=(\alpha^i)_{i\in\Ic}$ where $\alpha^i:\R_+\x\Om\x\Omt\to A$ is a predictable process  w.r.t. $(\Fc_s\ox\Fct_s)_{s\geq 0}$ .
 
 Fix $t\in[0,T]$ and $\mu=\sum_{i\in V} {\delta_{(i,x^i)}}\in\Ed$. For the sake of clarity, we omit the indices $(t,\mu)$ in the notations. Given $\alpha\in\tilde{\Ac}$, we define 
 \begin{equation*}
	\Z{n,\alpha}_s=\sum_{i\in V^{n,\alpha}_s} {\delta_{(i,\X{n,i}_s)}},\quad s\geq t,
 \end{equation*} 
 as the population process on the enlarged probability space corresponding to the branching parameters $\gamma_n$ and $(p_{n,k})_{0\leq k\leq n}$ and the diffusion
 \begin{equation*}
 	d\X{n,i}_s = b_n(\X{n,i}_s,\alpha^i_s) ds + \sigma_n(\X{n,i}_s,\alpha^i_s) dB^i_s + \frac{1}{\sqrt{n}} d\Bt^i_s.
 \end{equation*}
  In addition, we define the cost function $\bar{J}_n:[0,T]\x\Ed\x\tilde{\Ac}\to[0,1]$ by
  \begin{equation*}
    \bar{J}_n(t,\mu,\alpha) := \Et\left[\varGamma^{n,\alpha}_T \prod_{i\in V^{n,\alpha}_T} {g_n\left(\X{n,i}_T\right)}\right],
  \end{equation*}
  where $\Et$ denotes the expectation w.r.t. $\P\ox\Pt$ and 
  \begin{equation*}
  	\varGamma^{n,\alpha}_T := \exp{\left(-\int_t^T{\sum_{i\in V^{n,\alpha}_s} c_n\left(\X{n,i}_s,\alpha^i_s\right) \,ds }\right)}.
  \end{equation*}
  Similarly, we define both the value functions $\bar{v}_n:[0,T]\x\Ed\to[0,1]$ and $v_n:[0,T]\x\R^d\to[0,1]$ by
  \begin{equation*}
    \bar{v}_n(t,\mu):=\inf_{\alpha\in\tilde{\Ac}} {\bar{J}_n(t,\mu,\alpha)} \quad \text{and} \quad v_n(t,x):=\bar{v}_n(t,\delta_{(\emptyset,x)}).
  \end{equation*}

  In view of Corollary~\ref{cor:hjb-smooth}, the value function $v_n$ satisfies in the classical sense
  \begin{equation*}
    \partial_t v_n\left(t,x\right) + H_n\left(x,v_n(t,x),D_x{v_n}(t,x), D^2_x{v_n}(t,x)\right) = 0,\quad \forall\,\left(t, x\right)\in[0,T)\x\R^d,
  \end{equation*}
  where $H_n:\R^d\x[0,1]\x\R^d\x\R^{d\x d}\to\R$  is given by
  \begin{multline*}
   H_n(x,r,p,M) := \inf_{a\in A} \Bigg\{b_n(x,a)\cdot p + \frac{1}{2}\mathrm{tr}\left(\left(\sigma_n\sigma_n^*(x,a)+\frac{1}{n} I_d\right) M\right) \\
    + \gamma_n(x,a) \left(\sum_{k=0}^n {p_{n,k}(x,a)r^k} - r\right) - c_n(x,a) r\Bigg\},
  \end{multline*}
  By Lemma~\ref{lem:last} below, $v_n$ converges uniformly to $v$. In addition, since $\|b_n-b\|$, $\|\sigma_n-\sigma\|$, $\|\gamma_{n}-\gamma\|$, $\|p_{n,k}-p_k\|$, $\|c_n-c\|$ and $\|g_n-g\|$ vanishes as $n$ goes to $\infty$, one easily checks that $H_n$ converges locally uniformly to $H$ given by \eqref{eq:hamiltonian}. Hence, it follows from the stability of viscosity solutions (see,\eg, Lemma II.6.2 in~\cite{fleming06}) that $v$ is a viscosity solution of the HJB equation~\eqref{eq:hjb}.
  
 Similarly, in view of Corollary~\ref{cor:hjb-smooth}, it holds
 \begin{equation*}
 	\bar{v}_n(t,\mu) = \prod_{i\in V} {v_n(t,x^i)}.
 \end{equation*}
 Taking the limit $n\to\infty$, it follows from Lemma~\ref{lem:last} that the branching property~\eqref{eq:branching} is satisfied.
  
  \begin{lemma}\label{lem:last}
  	With the notations above, it holds
  	\begin{equation*}
  		\lim_{n\to\infty} {\sup_{i\in V} \sup_{x^i\in\R^d} \sup_{t\in[0,T]} \left|\bar{v}_n(t,\mu) - \bar{v}(t,\mu)\right|} = 0.
	\end{equation*}
  \end{lemma}
  
  \begin{proof}
    By abuse of notation, given $\alpha\in\tilde{\Ac}$, we denote by $\Z{\alpha}=\sum_{i\in V^{\alpha}} {\delta_{(i,\X{i})}}$ the solution of~\eqref{eq:eds} in the enlarged probability space. Let $\bar{J}_{\infty}:[0,T]\x\Ed\x\tilde{\Ac}\to[0,1]$ be given by
    \begin{equation*}
      \bar{J}_{\infty}(t,\mu,\alpha) := \Et\left[\varGamma^{\alpha}_T \prod_{i\in V^{\alpha}_T} {g\left(\X{i}_T\right)}\right].
    \end{equation*}
  	First, we observe that for all $\alpha\in\tilde{\Ac}$,
  	\begin{multline*}
  		\left|\bar{J}_{\infty}(t,\mu,\alpha) - \bar{J}_n(t,\mu,\alpha)\right| \leq \P\ox\Pt\left(\Om\x\tilde{\Om}\setminus F^n_\delta\right) \\
  		+ \Et\left[\left|\varGamma_T^{\alpha} - \varGamma_T^{n,\alpha}\right|\mathds{1}_{F^n_\delta}\right] + \Et\left[\left|\prod_{i\in V^{\alpha}_T} {g\left(\X{i}_T\right)} - \prod_{i\in V^{n,\alpha}_T} {g_n\left(\X{n,i}_T\right)}\right|\mathds{1}_{F^n_\delta}\right],
  	\end{multline*}
  	where 
  	\begin{equation*}
  		F^n_\delta := \left\{V_s^{\alpha} = V_s^{n,\alpha},\ \left|\X{i}_s-\X{n,i}_s\right|\leq \delta,\  i\in V^{\alpha}_s,\ s\in\left[t,T\right]\right\}.
  	\end{equation*}
  	Using $1-e^{-x}\leq x$, the second term on the r.h.s. can be bounded as follows:
  	\begin{equation*}
  		\Et\left[\left|\varGamma_T^{\alpha} - \varGamma_T^{n,\alpha}\right|\mathds{1}_{F^n_\delta}\right] \\
  		\leq \Et\left[\int_t^T{\sum_{i\in V^{\alpha}_s} \left|c\left(\X{i}_s,\alpha^i_s\right) - c_n\left(\X{n,i}_s,\alpha^i_s\right)\right| \,ds }\,\mathds{1}_{F^n_\delta}\right]
  	\end{equation*}
  	Regarding the third term, we start by observing that for all $k\in\N$, $(x_1,\ldots,x_k)$ and $(y_1,\ldots,y_k)$ in $[0,1]^k$,
  	\begin{equation*}
   		\left|\prod_{l=1}^k {x_l} - \prod_{l=1}^k {y_l}\right| \leq \sum_{l=1}^k {\left|x_l-y_l\right|}.
  	\end{equation*}
  	Then it follows that
  	\begin{equation*}
  		\Et\left[\left|\prod_{i\in V^{\alpha}_T} {g\left(\X{i}_T\right)} - \prod_{i\in V^{n,\alpha}_T} {g_n\left(\X{n,i}_T\right)}\right|\mathds{1}_{F^n_\delta}\right] \leq \Et\left[\sum_{i\in V^{\alpha}_T} \left|g\left(\X{i}_s\right) - g_n\left(\X{n,i}_s\right)\right| \mathds{1}_{F^n_\delta}\right].
  	\end{equation*}
  	Given $\varepsilon>0$, we take $\delta>0$ such that $|c(y,a)-c(y',a)|\leq\varepsilon$ and $|g(y)-g(y')|\leq\varepsilon$ for all $|y-y'|\leq\delta$ and $a\in A$. We deduce that
  	\begin{multline*}
  		\left|\bar{J}_{\infty}(t,\mu,\alpha) - \bar{J}_n(t,\mu,\alpha)\right| \leq \P\ox\Pt\left(\Om\x\tilde{\Om}\setminus F^n_\delta\right) \\
  		+ C T \left(\|c_n - c\| + \varepsilon\right) + C \left(\|g_n - g\| + \varepsilon\right),
  	\end{multline*}
  	where $C = |V| e^{\gammab M T}$. Further, Proposition~\ref{prop:perturbation} ensures that $\P\ox\Pt\left(\Om\x\tilde{\Om}\setminus F^n_\delta\right)$ vanishes, uniformly w.r.t. $t$, $(x^i)_{i\in V}$ and $\alpha$, as $n$ tends to infinity. We deduce that 
  	\begin{equation*}
  		\lim_{n\to\infty} {\sup_{i\in V} \sup_{x^i\in\R^d} \sup_{t\in[0,T]} \sup_{\alpha\in\tilde{\Ac}} {\left|\bar{J}_{\infty}(t,\mu,\alpha) - \bar{J}_n(t,\mu,\alpha)\right|}} = 0.
  	\end{equation*}
  	It follows that 
  	\begin{equation*}
  		\lim_{n\to\infty} {\sup_{i\in V} \sup_{x^i\in\R^d} \sup_{t\in[0,T]} {\left|\bar{v}_{\infty}(t,\mu) - \bar{v}_n(t,\mu)\right|}} = 0,
  	\end{equation*}
  	where 
  	\begin{equation*}
  		\bar{v}_{\infty}(t,\mu) = \inf_{\alpha\in\tilde{\Ac}} {\bar{J}_{\infty}(t,\mu,\alpha)}.
  	\end{equation*}
  	To conclude, it remains to show that $\bar{v}$ coincide with $\bar{v}_{\infty}$. Given $\alpha\in\tilde{\Ac}$ and $\omt\in\Omt$, we define $\alpha^{\omt}:\R_+\x\Om\to A$ by $\alpha^{\omt}_s(\om):=\alpha(s,\om,\omt)$. It is clear that, $\omt\in\Omt$ being fixed, $\alpha^{\omt}\in\Ac$. 	Since $d\Bt^{i}$ appears with coefficient $0$ in~\eqref{eq:edsi} for all $i\in\Ic$, it follows that 
  	\begin{equation*}
  		\bar{J}_{\infty}(t,\mu,\alpha) = \int_{\Omt}{\bar{J}(t,\mu,\alpha^{\omt}) \,\Pt(d\omt)} \geq \bar{v}(t,\mu).
  	\end{equation*}
  	We deduce that $\bar{v}_{\infty}\geq \bar{v}$. The other inequality is obvious by natural injection of $\Ac$ into $\tilde{\Ac}$.
  \end{proof}
  
  \section{Dynamic programming principle}\label{sec:dpp}
  
  The aim of this Section is to derive the DPP satisfied by the value function. Recall that, under Assumption~\ref{hyp:smooth}, the DPP~\eqref{eq:PPD_hard}--\eqref{eq:PPD_easy} holds by Corollary~\ref{cor:hjb-smooth}. For the general case, we use the same approximation procedure 
  described in Section~\ref{sec:hjbexist}. Note that the formulation of the DPP differs from Proposition~\ref{prop:dpp-smooth} as we need to work in the enlarged probability space.
  
  \begin{theorem}
  	With the notation of Section~\ref{sec:hjbexist}, denote by $\tilde{\Tc}_{t,T}$ the collection of all stopping times w.r.t. $(\Fc_s\ox\tilde{\Fc}_s)_{s\geq 0}$ valued in $[t,T]$. For all $t\in[0,T]$ and $\mu=\sum_{i\in V}{\delta_{(i,x^i)}}\in\Ed$, it holds:\\
  	\rmi for all $\alpha\in\tilde{\Ac}$ and $\tau\in \tilde{\Tc}_{t,T}$,
  	\begin{equation*}
   		\prod_{i\in V} {v\left(t,x^i\right)} \leq \Et\left[\varGamma^{t,\mu,\alpha}_\tau \prod_{i\in V^{t,\mu,\alpha}_{\tau}} {v\left(\tau,X^{i}_{\tau}\right)}\right];
  	\end{equation*}
  	\rmii for all $\varepsilon>0$, there exists $\alpha\in\tilde{\Ac}$ such that, for all $\tau\in \tilde{\Tc}_{t,T}$,
  	\begin{equation*}
   		\prod_{i\in V} {v\left(t,x^i\right)} + \varepsilon \geq \Et\left[\varGamma^{t,\mu,\alpha}_\tau \prod_{i\in V^{t,\mu,\alpha}_{\tau}} {v\left(\tau,X^{i}_{\tau}\right)}\right].
  	\end{equation*}
  \end{theorem}
  
  \begin{remark}
   In view of the branching property~\eqref{eq:branching}, the formulation of the DPP above is equivalent to
  \begin{equation*}
    \bar{v}\left(t,\mu\right) 
      = \inf_{\alpha\in\tilde{\Ac}} { \inf_{\tau\in \tilde{\Tc}_{t,T}} {\Et\left[\bar{v}\left(\tau,\Z{t,\mu,\alpha}_\tau\right)\right]}}
      = \inf_{\alpha\in\tilde{\Ac}} { \sup_{\tau\in \tilde{\Tc}_{t,T}} {\Et\left[\bar{v}\left(\tau,\Z{t,\mu,\alpha}_\tau\right)\right]}}.
  \end{equation*}
  \end{remark}

  \begin{proof}
   Once again, we omit the indices $(t,\mu)$ in the notations. We start by proving~(i). By Corollary~\ref{cor:hjb-smooth}, we have for all $\alpha\in\tilde{\Ac}$ and $\tau\in \tilde{\Tc}_{t,T}$,
   \begin{equation}\label{eq:dpp-proof-hard}
    \prod_{i\in V} {v_n\left(t,x^i\right)} \leq \Et\left[\varGamma^{n,\alpha}_\tau \prod_{i\in V^{n,\alpha}_{\tau}} {v_n\left(\tau,X^{n,i}_{\tau}\right)}\right].
   \end{equation}
   Further, it holds
   \begin{multline*}
  	\Et\left[\left|\varGamma^{\alpha}_\tau \prod_{i\in V^{\alpha}_{\tau}} {v\left(\tau,X^{i}_{\tau}\right)} - \varGamma^{n,\alpha}_\tau \prod_{i\in V^{n,\alpha}_{\tau}} {v_n\left(\tau,X^{n,i}_{\tau}\right)}\right|\right] \leq \P\ox\Pt\left(\Om\x\tilde{\Om}\setminus F^n_\delta\right) \\
  	+ \Et\left[\left|\varGamma_\tau^{\alpha} - \varGamma_\tau^{n,\alpha}\right|\mathds{1}_{F^n_\delta}\right] + \Et\left[\left|\prod_{i\in V^{\alpha}_\tau} {v\left(\tau,\X{i}_\tau\right)} - \prod_{i\in V^{n,\alpha}_\tau} {v_n\left(\tau,\X{n,i}_\tau\right)}\right|\mathds{1}_{F^n_\delta}\right],
   \end{multline*}
   where 
   \begin{equation*}
  	F^n_\delta := \left\{V_s^{\alpha} = V_s^{n,\alpha},\ \left|\X{i}_s-\X{n,i}_s\right|\leq \delta,\  i\in V^{\alpha}_s,\ s\in\left[t,T\right]\right\}.
   \end{equation*}
   Using the same arguments as in the proof of Lemma~\ref{lem:last}, we derive that 
   \begin{equation*}
    \lim_{n\to+\infty} \Et\left[\varGamma^{n,\alpha}_\tau \prod_{i\in V^{n,\alpha}_{\tau}} {v_n\left(\tau,X^{n,i}_{\tau}\right)} \right] = \Et\left[\varGamma^{\alpha}_\tau \prod_{i\in V^{\alpha}_{\tau}} {v\left(\tau,X^{i}_{\tau}\right)}\right].
   \end{equation*}
   Notice that these arguments allow to prove that the convergence is uniform w.r.t. $\tau\in\tilde{\Tc}_{t,T}$. 
   To conclude, it remains to take the limit $n\to+\infty$ in~\eqref{eq:dpp-proof-hard}. Let us turn now to the proof of (ii). Fix $\varepsilon>0$. In view of the above, we choose $n\in\N$ sufficiently large to ensure that  for all $\tau\in\tilde{\Tc}_{t,T}$, 
   \begin{equation*}
    \Et\left[\left|\varGamma^{\alpha}_\tau \prod_{i\in V^{\alpha}_{\tau}} {v\left(\tau,X^{i}_{\tau}\right)} - \varGamma^{n,\alpha}_\tau \prod_{i\in V^{n,\alpha}_{\tau}} {v_n\left(\tau,X^{n,i}_{\tau}\right)}\right|\right] \leq \frac{\varepsilon}{3}.
   \end{equation*}
   Then, by Corollary~\ref{cor:hjb-smooth}, we take $\alpha\in\tilde{\Ac}$ such that for all $\tau\in\Tc_{t,T}$, 
   \begin{equation*}
   	\prod_{i\in V} {v_n\left(t,x^i\right)} + \frac{\varepsilon}{3} \geq \Et\left[\varGamma^{n,\alpha}_\tau \prod_{i\in V^{n,\alpha}_{\tau}} {v_n\left(\tau,X^{n,i}_{\tau}\right)}\right].
   \end{equation*}
   We conclude as follows:
   \begin{align*}
    \prod_{i\in V} {v\left(t,x^i\right)} \geq \prod_{i\in V} {v_n\left(t,x^i\right)} - \frac{\varepsilon}{3}
      & \geq \Et\left[\varGamma^{n,\alpha}_\tau \prod_{i\in V^{n,\alpha}_{\tau}} {v_n\left(\tau,X^{n,i}_{\tau}\right)}\right] - \frac{2\varepsilon}{3} \\
      & \geq \Et\left[\varGamma^{\alpha}_\tau \prod_{i\in V^{\alpha}_{\tau}} {v\left(\tau,X^{i}_{\tau}\right)}\right] - \varepsilon.
   \end{align*}
  \end{proof}

 \subparagraph{Acknowledgements.}
 I gratefully acknowledge my PhD supervisors Nicolas Champagnat and Denis Talay for supervising this work. 
 I am also grateful to the financial support of ERC Advanced Grant 321111 ROFIRM.

 \bibliographystyle{habbrv}
 \bibliography{These}{}
 
\end{document}